\documentclass[reqno, a4paper, 10pt]{amsart}
\usepackage{amssymb,url, color, mathrsfs}
\usepackage[colorlinks=true, bookmarks=true, pdfstartview=FitH, pagebackref=true, linktocpage=true, linkcolor = magenta, citecolor = blue]{hyperref}
\usepackage[short,nodayofweek]{datetime}
\usepackage[pagewise,running,mathlines,displaymath, switch]{lineno}

\theoremstyle{plain}
\numberwithin{equation}{section}
\newtheorem{theorem}{Theorem}[section]
\newtheorem{lemma}[theorem]{Lemma}
\newtheorem{corollary}[theorem]{Corollary}
\newtheorem{proposition}[theorem]{Proposition}
\theoremstyle{remark}
\newtheorem{claim}{\bf Claim}

\DeclareMathOperator{\B}{\mathbb{B}}
\DeclareMathOperator{\G}{\mathbb{G}}
\DeclareMathOperator{\Rset}{\mathbf{R}}

\DeclareMathOperator{\Hcirc}{\mathop {\textit H}\limits^ \circ\mkern-1.5mu {}^1}

\def\util{{\widetilde u}}
\def\htil{{\widetilde h}}

\hypersetup{
pdftitle={MT mean zero} 
pdfauthor={Van Hoang Nguyen},
colorlinks = true,
linkcolor = magenta,
citecolor = blue,
}

\def\R{\mathbf R}% tap so thuc

\def\NN{{\mathsf N}}
\def\DD{{\mathsf D}}

\def\al{\alpha}% alpha
\def\Om{\Omega}% Omega
\def\be{\beta}% beta
\def\de{\delta}% delta
\def\lam{\lambda}% lambda
\def\vphi{\varphi}% varphi
\def\ep{\epsilon}% epsilon
\def\na{\nabla}% nabla

\def\pa{\partial}% dao ham rieng
\def\lt{\left}% trai
\def\rt{\right}% phai
\def\o{\overline}

\numberwithin{equation}{section}

\title[Moser--Trudinger inequality with mean value zero in $\R^2$]{An improved Moser--Trudinger inequality involving the first non-zero Neumann eigenvalue with mean value zero in $\R^2$}

\def\cfac#1{\ifmmode\setbox7\hbox{$\accent"5E#1$}\else\setbox7\hbox{\accent"5E#1}\penalty 10000\relax\fi\raise 1\ht7\hbox{\lower1.05ex\hbox to 1\wd7{\hss\accent"13\hss}}\penalty 10000\hskip-1\wd7\penalty 10000\box7 }

\author[Q.A. Ng\^o]{Qu\cfac oc Anh Ng\^o}
\address[Q.A. Ng\^o]{Department of Mathematics\\
College of Science, Vi\^{e}t Nam National University\\
H\`{a} N\^{o}i, Vi\^{e}t Nam.}
\email{\href{mailto: Q.A. Ng\^o <nqanh@vnu.edu.vn>}{nqanh@vnu.edu.vn}}
\email{\href{mailto: Q.A. Ng\^o <bookworm\_vn@yahoo.com>}{bookworm\_vn@yahoo.com}}

\author[V.H. Nguyen]{Van Hoang Nguyen}
\address[V.H. Nguyen]{Institut de Math\'ematiques de Toulouse\\
Universit\'e Paul Sabatier\\
31062 Toulouse c\'edex 09, France.}
\email{\href{mailto: V.H. Nguyen <van-hoang.nguyen@math.univ-toulouse.fr>}{van-hoang.nguyen@math.univ-toulouse.fr}}

\begin{document}

\allowdisplaybreaks

\begin{abstract}
Let $\Omega$ be a smooth bounded domain in $\Rset^2$ and $\lambda^\NN (\Omega)$ the first non-zero Neumann eigenvalue of the operator $-\Delta$ on $\Omega$. In this paper, for any $\gamma \in [0, \lambda^\NN (\Omega) )$, we establish the following improved Moser--Trudinger inequality
\[
\sup_{u} \int_{\Omega} e^{2\pi u^2} dx < +\infty
\]
for arbitrary functions $u$ in $H^1(\Omega)$ satisfying $\int_\Omega u dx =0$ and $\|\nabla u\|_2^2 -\alpha \|u\|_2^2 \leqslant 1$. Furthermore, this supremum is attained by some function $u^*\in H^1(\Omega)$. This strengthens the results of Chang and Yang (\textit{J. Differential Geom.} {\bf 27} (1988) 259--296) and of Lu and Yang (\textit{Nonlinear Anal.} {\bf 70} (2009) 2992--3001).
\end{abstract}

\date{\bf \today \; at \, \currenttime}

\subjclass[2010]{46E35, 26D10}

\keywords{Moser--Trudinger inequality, blow-up analysis, sharp constant, extremal functions, regularity theory}

\maketitle

\section{Introduction}

On a smooth bounded domain $\Omega$ in $\R^n$ with $n \geqslant 2$, the classical Sobolev inequality tells us that there is a continuous embedding $W^{k,p} (\Omega) \hookrightarrow L^q(\Omega)$ for all $1 \leqslant q \leqslant np/(n-kp)$ provided $p <n/k$. Here $W^{k,p}(\Omega)$ is the usual Sobolev space constructed as the completion of $C_0^\infty(\Om)$ under a suitable norm. However, in the borderline case $p=n/k$, the continuous embedding $W^{k,p} (\Omega) \hookrightarrow L^\infty (\Omega)$ is no longer available by some easy examples. In this case, the so-called Moser--Trudinger inequality is a perfect replacement. 

This inequality, in the form due to Trudinger \cite{T1967}, asserts that
\begin{equation}\label{eq:T}
\sup_{u\in W^{1,n}(\Om), \|\na u\|_n \leqslant 1} \int_\Om \exp \big( \gamma |u|^{\frac n{n-1}} \big) dx < +\infty
\end{equation}
for some non-negative constant $\gamma$. In \eqref{eq:T} we denote by $\| \cdot \|_p$ the usual $L^p$-norm. The mathematical meaning of \eqref{eq:T} is that the Sobolev space $W^{1,n}(\Omega)$ can be continuously embedded into the Orlicz space associated with the Young function $\exp (t^{n/(n-1)})-1$. As remarked in \cite{Cianchi}, such an embedding was announced, without proof, by Yudovi\v c \cite{Y1961} and independently was proved, in a slightly weaker form, by Poho\v zaev \cite{P1965, P1965e}. After the seminal work \cite{T1967}, a lot of generalizations and improvements of \eqref{eq:T}, including the exhibition of the largest constant $\gamma$ in which the inequality \eqref{eq:T} still holds, have been made.

In fact, one cannot expect that there is no upper bound for $\gamma$ in which \eqref{eq:T} holds. The problem of specifying such an upper bound, for functions $u$ belonging in the subspace $W_0^{1,n}(\Omega)$ of $W^{1,n}(\Omega)$ was completely solved by Moser. In \cite{M1970}, it was proved that
\begin{equation}\label{eq:MT}
\sup_{u\in W_0^{1,n}(\Om), \|\na u\|_n \leqslant 1} \int_\Om \exp \big( \gamma |u|^{\frac n{n-1}} \big) dx < +\infty
\end{equation}
for any $\gamma \leqslant \gamma_n := n\omega_{n-1}^{1/(n-1)}$. Here the subspace $W_0^{1,n}(\Om)$ is the completion of $C^\infty(\Omega)$ in $W^{1,n} (\Omega)$ and by $\omega_{n-1}$ we mean the area of the unit $(n-1)$-sphere in $\R^n$. (In the special case $n=2$, we simply denote $W^{1,n}(\Omega)$ by $H^1(\Omega)$ and $W_0^{1,n}(\Omega)$ by $H_0^1(\Omega)$ for simplifying notation.) Here the constant $\gamma_n$ is sharp and by the sharp constant $\gamma_n$ we mean the left hand side of \eqref{eq:MT} becomes infinity if $\gamma > \gamma_n$. A sharp version of \eqref{eq:MT} for higher order derivatives, meaning that the the following inequality
\begin{equation}\label{eq:AMT}
\sup_{u\in W_0^{m,n/m}(\Omega), \|\nabla^m u\|_{n/m} \leqslant 1} \int_\Om \exp \big( \gamma_{n,m} |u|^{\frac n{n-m}} \big) dx < +\infty
\end{equation}
%Sobolev space $W_0^{1,n}(\Omega)$ in \eqref{eq:MT} is replaced by $W_0^{m,n/m}(\Omega)$, 
with a sharp constant $\gamma_{n,m}$ with $n>m$, was established by Adams \cite{Adams}.

It is now widely recognized that the Moser--Trudinger inequality \eqref{eq:MT} as well as the Adams inequality \eqref{eq:AMT} and their variants have a strong impact in studying nonlinear partial differential equations. Although the inequality \eqref{eq:MT} has rapidly captured attention and a number of generalizations have been done, however, a prior to a work due to Chang and Yang \cite{CYzero}, all known results involving \eqref{eq:MT} are essentially limited to functions vanishing on $\partial \Omega$.

In \cite{CYzero}, limited to the two-dimensional case, Chang and Yang proved a sharp Moser--Trudinger inequality for functions in $H^1(\Om)$ with mean value zero as follows
\begin{equation}\label{eq:ChangYang}
\sup_{u\in H^1(\Om), \int_\Om u dx =0, \|\na u\|_2 \leqslant 1} \int_\Om \exp (\gamma u^2) dx < +\infty,
\end{equation}
for any $\gamma \leqslant 2\pi$. Moreover, the constant $2\pi$ is sharp in the sense that if $\gamma > 2\pi$, then the supremum in \eqref{eq:ChangYang} is infinity. A generalization of \eqref{eq:ChangYang} to arbitrary dimension was proved by Cianchi \cite{Cianchi} by using an asymptotically sharp relative isoperimetric inequality for domains in $\R^n$.

The motivation of writing this paper traces back to the two works by Lu and Yang in \cite{LYzero} and by Yang in \cite{Yang15}. However, before we mention the main result in \cite{LYzero}, let us first recall an interesting result due to Adimurthi and Druet. In \cite{AD2004}, the authors essentially improve \eqref{eq:MT} with $n$ replaced by $2$ by showing that the inequality
\begin{equation}\label{eq:AD}
\sup_{u\in H_0^1(\Omega), \|\na u\|_2 \leqslant 1} \int_\Om \exp \big(4\pi u^2 ( 1 + \alpha \|u\|_2^2) \big) dx < +\infty
\end{equation}
holds for any $\alpha \in [0, \lambda^\DD (\Omega))$ where $\lambda^\DD (\Omega)$ is the first (non-zero) Dirichlet eigenvalue of the operator $-\Delta$ on $\Omega$. 

In \cite{LYzero}, Lu and Yang essentially sharpened the Chang and Yang inequality \eqref{eq:ChangYang} in the spirit of the Adimurthi and Druet inequality \eqref{eq:AD}. To understand their generalization, let us first denote by $\Hcirc (\Omega)$ a close subspace of $H^1(\Omega)$. given by
\[
\Hcirc (\Omega) = \Big\{ u \in H^1(\Omega) : \int_\Omega udx = 0 \Big\}.
\]
We also denote by
\[
\lambda^\NN(\Omega) = \inf\Big\{\|\na u\|_2^2\,:\, u\in \Hcirc (\Om), \int_\Omega |u|^2 dx =1 \Big\}
\]
the first non-zero Neumann eigenvalue of the operator $-\Delta$ on $\Omega$. We also let $q(t) =1 + a_1t +\cdots + a_k t^k$ be a polynomial of order $k$ with coefficients satisfying
\begin{equation}\label{eq:coeff}
0\leqslant a_1 < \lambda^\NN (\Om),\quad 0\leqslant a_2 \leqslant \lambda^\NN (\Om) a_1, \quad \ldots,\quad a_k \leqslant \lambda^\NN(\Om) a_{k-1}.
\end{equation} 
The main result in \cite{LYzero} is to establish the following inequality
\begin{equation}\label{eq:LuYang}
\sup_{u\in H^1(\Om), \|\na u\|_2 \leqslant 1, \int_\Om u dx =0} \int_\Om \exp \big( 2\pi u^2 q(\|u\|_2^2) \big) dx < +\infty.
\end{equation}
Moreover, if the first coefficient $a_1 \geqslant \lam(\Om)$, then the supremum in \eqref{eq:LuYang} will be infinite for any choice of other coefficients $a_2$,..., $a_k$. Clearly the inequality \eqref{eq:LuYang} is an improvement of the Chang--Yang inequality \eqref{eq:ChangYang} in spirit of Adimurthi and Druet \cite{AD2004} for the Moser--Trudinger inequality \eqref{eq:MT}. Such an improvement was recent proved for Moser--Trudinger inequality in whole space $\R^n$ by do \'O and de Souza \cite{doO2014,doOSouza} and for sharp Adams inequality in dimension four by Lu and Yang \cite{LYadams}. It was also proved in \cite{LYzero} that there exists $0< \varepsilon _0 \leqslant \lambda(\Om)$ such that the supremum in \eqref{eq:LuYang} is attained for any $0 \leqslant a_1 < \varepsilon _0$. In particular, there exists extremal functions for \eqref{eq:ChangYang}. For more about the existence of extremal functions for Moser--Trudinger inequality \eqref{eq:MT} and its generalization, we refer reader to \cite{CC1986,Csato2015,Csato2016,Flucher1992,Li2001,Lin1996,Yang06,Yang07} and references therein.

Let us now discuss Yang's results in \cite{Yang15}. Among other things, for each $\alpha \in [0, \lambda^\DD (\Om) )$ fixed, by introducing an equivalent norm on $H^1(\Omega)$ being orthogonal to constant functions,
\[
\|u\|_{1,\al}^2 = \|\na u\|_2^2 -\alpha \|u\|_2^2,
\]
thanks to the Poincar\'e inequality, the following inequality in spirit of Adimurthi and Druet \cite{AD2004} and Tintarev \cite{Tintarev} 
\begin{equation}\label{eq:YAD}
\sup_{u\in H_0^1(\Omega), \|u\|_{1,\al} \leqslant 1} \int_\Om \exp \big(4\pi u^2 \big) dx < +\infty
\end{equation}
was proved; see \cite[Theorem 1]{Yang15}. Furthermore, the supremum in \eqref{eq:YAD} can be attained by some function.

In this note, we aim to prove another improvement of \eqref{eq:LuYang} in the same fashion of the Yang inequality \eqref{eq:YAD}. Still using the norm $\|\cdot\|_{1,\al}$ on the subspace of $H^1(\Om)$ being orthogonal to constant functions, our first result reads as follows.

\begin{theorem}\label{Main1}
Let $\Om$ be a smooth bounded domain in $\R^2$ and $0 \leqslant \al < \lam^\NN(\Om)$. There holds
\begin{equation}\label{eq:Mainresult1}
\sup_{u\in \Hcirc (\Om), \|u\|_{1,\alpha} \leqslant 1} \int_\Om e^{2\pi u^2} dx < +\infty.
\end{equation}
\end{theorem}

Clearly, our inequality \eqref{eq:Mainresult1} in Theorem \ref{Main1} implies the Chang and Yang \eqref{eq:ChangYang} when $\alpha = 0$; hence it is an improvement of \eqref{eq:ChangYang}. We note that generalizations in the fashion of Theorem \ref{Main1} have already existed in the literature. For instance, Yang and Zhu \cite{YZ} proved a similar result for a singular Moser--Trudinger inequality in dimension two and a similar result for the Adams inequality \eqref{eq:AMT} in dimension four was also proved by the second author in \cite{Nguyen17}. %Hence, this paper could be thought of a lower-order case compared with \cite{Nguyen17}.

Next, we would like to compare our inequality \eqref{eq:Mainresult1} and the Lu and Yang inequality \eqref{eq:LuYang}. As shown in \cite[Section \S6]{Nguyen17}, for any choice of $a_1,\ldots, a_k$ satisfying \eqref{eq:coeff}, we always can choose some small number $\alpha \in (0, \lam^\NN (\Omega))$ such that 
$$q(\|u\|_2^2) \leqslant 1/(1 -\alpha \|u\|_2^2)$$ 
for any $u\in H^1(\Om)$ satisfying $\int_\Om u dx =0$ and $\|\na u\|_2 \leqslant 1$. Simply choosing $v = u (1 -\alpha \|u\|_2^2)^{-1/2}$, we deduce that $v \in \Hcirc (\Om)$ and that $\|\na v\|_{1,\alpha} \leqslant 1$. However, $u^2 q(\|u\|_2^2) \leqslant v^2$. Therefore our inequality \eqref{eq:Mainresult1} is indeed stronger than the one of Lu and Yang \eqref{eq:LuYang}. Before going to an other result, let us mention the following corollary.

\begin{corollary}\label{Main3}
Let $\Om$ be a smooth bounded domain in $\R^2$ and $0 \leqslant \al < \lam^\NN(\Om)$. Then there exists some constant $C>0$ such that for all $u \in H^1(\Omega)$, there holds
\begin{equation}\label{eq:Mainresult3}
\log \Big( \int_\Om e^ u dx \Big) \leqslant \frac 1{8\pi} \int_\Omega |\nabla u|^2 dx - \frac \alpha {8\pi} \int_\Omega u^2 dx + \frac 1{|\Omega|} \int_\Omega u dx + C.
\end{equation}
\end{corollary}

Next we discuss our second result concerning to extremal functions for \eqref{eq:Mainresult1}. We shall prove the following.

\begin{theorem}\label{Main2}
Let $\Om$ be a smooth bounded domain in $\R^2$ and let $\alpha \in [0, \lambda^\NN (\Omega) )$. Then there exists a function $u^* \in H^1(\Om)$ satisfying $\int_\Om u^* dx =0$ and $\|u^*\|_{1,\al} =1$ such that
\[
\int_\Om e^{2\pi {u^*}^2} dx = \sup_{u\in \Hcirc (\Om), \|u\|_{1,\alpha} \leqslant 1} \int_\Om e^{2\pi u^2} dx < +\infty;
\]
that is, the supremum in \eqref{eq:Mainresult1} is attained by $u^*$.
\end{theorem}

The proof of Theorems \ref{Main1} and \ref{Main2} is based on blow-up analysis. For interested readers on this method, we refer to the book \cite{Druet}; see also \cite{AD2004, Li2001, Lin1996, Nguyen17, Yang06, Yang07, YZ} for more detail on this technique. It is important to note that, unlike the case treated in \cite{Nguyen17}, in our situation, the blow-up behavior can occur on the boundary $\pa \Om$ as in \cite{LYzero}, which makes the analysis more difficult and rather involved.

The organization of this paper is as follows. In the next section \S\ref{sec-ExtremalFun} we prove a subcritical version of \eqref{eq:Mainresult1} as well as the existence of extremal functions for this subcritical inequality. Then in order to prove the critical inequality, we analyze the asymptotic behavior of the sequence of extremal functions for the subcritical inequality in section \S\ref{sec-Asymptotic} and establish some capacity estimates in section \S\ref{sec-Capacity}, which eventually lead us to the proof of Theorems \ref{Main1} and \ref{Main2} in section \S\ref{sec-Proofs}. Finally, we prove Corollary \ref{Main3} in subsection \S\ref{subsec-Application} and provide an application of Corollary \ref{Main3} in section \S\ref{subsec-Application}; see Theorem \ref{thmApplication}.

%%%%%%%%%%%%
%%%%%%%%%%%%
%%%%%%%%%%%%
%%%%%%%%%%%%

\section{Extremal functions for the subcritical case}
\label{sec-ExtremalFun}

In this section, we study a subcritical Moser--Trudinger inequality for functions with mean value zero in $H^1(\Om)$. For each $0< \varepsilon < 2\pi$, we denote
\[
C_\varepsilon = \sup_{u\in \Hcirc (\Om), \|u\|_{1,\alpha} \leqslant 1} \int_\Om \exp ((2\pi-\varepsilon ) u^2 ) dx.
\]
Our main result in this section is the following.

\begin{proposition}\label{subcritical}
Let $\Om$ be a smooth bounded domain in $\R^2$ and let $0\leqslant \al < \lambda^\NN (\Omega) $. Then for any $0< \varepsilon < 2\pi$, we have that $C_\varepsilon < +\infty$ and that there exists $u_\varepsilon \in \Hcirc(\Om) \cap C^\infty(\overline \Om)$ such that $\|u_\varepsilon \|_{1,\al}=1$ and 
\begin{equation}\label{eq:sub}
C_\varepsilon = \int_\Om e^{(2\pi-\varepsilon ) u_\varepsilon ^2} dx.
\end{equation}
The Euler--Lagrange equation of $u_\varepsilon $ is given by
\begin{equation}\label{eq:EL}
\left\{
\begin{split}
-\Delta u_\varepsilon &= \lam_\varepsilon ^{-1} e^{\al_\varepsilon u_\varepsilon ^2} u_\varepsilon + \al u_\varepsilon -\lam_\varepsilon ^{-1} \mu_\varepsilon \qquad \mbox{ in } \Om,\\
\frac{\pa u_\varepsilon }{\pa \nu} &= 0 \qquad \mbox{ on } \pa \Om,\\
\|\na u_\varepsilon \|_{1,\al} &=1, \\
\al_\varepsilon &=2\pi -\varepsilon ,\\
\mu_\varepsilon &= \frac{1}{|\Om|}\int_\Om e^{\al_\varepsilon u_\varepsilon ^2} u_\varepsilon dx,\\
 \lam_\varepsilon &= \int_\Om e^{\al_\varepsilon u_\varepsilon ^2} u_\varepsilon ^2 dx.
\end{split}
\right.
\end{equation}
Furthermore, there holds
\begin{equation}\label{eq:liminf}
\liminf_{\varepsilon \to 0} \lam_\varepsilon >0.
\end{equation}
\end{proposition}

In the proof of Proposition \ref{subcritical}, inspired by \cite{Lions1985}, we need the following Lions-type concentration--compactness principle for functions in $H^1(\Om)$ with mean value zero.

\begin{lemma}\label{Lions}
Let $\{u_j\}_j\subset \Hcirc (\Om)$ such that $\|u_j\|_{1,\al} =1$ and $u_j \rightharpoonup u_0$ in $H^1(\Om)$ then for any $0 < p< 1/(1-\|u_0\|_{1,\al}^2)$, there holds
\[
\limsup_{j\to\infty} \int_\Om e^{2\pi p u_j^2}dx < +\infty.
\]
\end{lemma}
\begin{proof}
By the Poincar\'e inequality, we have $\|u_j\|_2^2 \leqslant \lambda^\NN (\Omega) ^{-1} \|\na u_j\|_2^2$, hence we get 
\[
\|\na u_j\|_2^2 \leqslant \frac{ \lambda^\NN (\Omega) }{ \lambda^\NN (\Omega) -\al}
\]
for any $j$. Consequently, the sequence $\{u_j\}_j$ is bounded in $H^1(\Om)$. Up to a subsequence, we assume, in addition, that $u_j\to u_0$ in $L^q(\Om)$ for any $1 \leqslant q < +\infty$ and $u_j\to u_0$ a.e. in $\Om$. We have
\[
\|\na u_j -\na u_0\|_2^2 =\|\na u_j\|_2^2 -\|\na u_0\|_2^2 +o(1) = 1-\|u_0\|_{1,\al}^2 + o(1).
\]
Thus, for any $p < 1/(1-\|u_0\|_{1,\al}^2)$, there exists $j_0$ such that $p\|\na (u_j-u_0)\|_2^2 \leqslant (p+1)/2 <1$ for any $j\geqslant j_0$ hence our conclusion is a consequence of the inequality of Chang and Yang \eqref{eq:ChangYang} and the elementary inequality $ab \leqslant \gamma a^2 + b^2/(4\gamma)$ for any $\gamma >0$.
\end{proof}

\begin{proof}[Proof of Proposition \ref{subcritical}]
Let $\{u_j\}_j$ be a maximizing sequence for $C_\varepsilon $. Under the condition $\int_\Om u_j dx =0$ and by using the Poincar\'e inequality as in proof of Lemma \ref{Lions} above, we see that $\{u_j\}_j$ is bounded in $H^1(\Om)$. Thus we can assume, in addition, that $u_j \rightharpoonup u_\varepsilon $ weakly in $H^1(\Om)$, $u_j\to u_\varepsilon $ in $L^q(\Om)$ for any $1 \leqslant q < +\infty$, and $u_j\to u_\varepsilon $ a.e. in $\Om$. If the limit function $u_\varepsilon \equiv 0$, then by Lemma \ref{Lions}, we can choose $1 < p < 2\pi/\al_\varepsilon $ in such a way that $\{\exp (\al_\varepsilon p u_j^2) \}_j$ is bounded in $L^1(\Om)$, which implies that
\[
C_\varepsilon = \lim_{j\to \infty} \int_\Om e^{\al_\varepsilon u_j^2} dx =|\Om|,
\]
which is impossible. Hence $u_\varepsilon \not\equiv 0$. By Lemma \ref{Lions}, we can choose $1< p < 1/(1 -\|u_\varepsilon \|_{1,\al}^2)$ such that $\{ \exp (\al_\varepsilon p u_j^2 ) \}_j$ is bounded in $L^1(\Om)$, hence
\[
C_\varepsilon = \lim_{j\to \infty} \int_\Om e^{\al_\varepsilon u_j^2} dx = \int_\Om e^{\al_\varepsilon u_\varepsilon ^2} dx.
\]
Obviously, we have $\int_\Om u_\varepsilon dx =0$. By the lower semi-continuous, we have $\|u_\varepsilon \|_{1,\al} \leqslant 1$. If $\|u_\varepsilon \|_{1,\al} < 1$, then we easily get a contradiction because
\[
C_\varepsilon %= \int_\Om \exp (\al_\varepsilon u_\varepsilon ^2 ) dx 
=\int_\Om \exp \Big(\al_\varepsilon \|u_\varepsilon \|_{1,\al}^2 \frac{u_\varepsilon ^2}{\|u_\varepsilon \|_{1,\al}^2} \Big) dx < \int_\Om \exp \Big(\al_\varepsilon \frac{u_\varepsilon ^2}{\|u_\varepsilon \|_{1,\al}^2} \Big) dx \leqslant C_\varepsilon .
\]
This shows that $\|u_\varepsilon \|_{1,\al} =1$ and hence $u_\varepsilon $ is a maximizer for $C_\varepsilon $. A straightforward computation shows that the Euler--Lagrange equation of $u_\varepsilon $ is given by \eqref{eq:EL}. By standard elliptic theory \cite{GT}, we get from \eqref{eq:EL} that $u_\varepsilon \in C^\infty(\overline \Om)$. To prove \eqref{eq:liminf}, we use the inequality $e^t \leqslant 1 + te^t$ for any $t \geqslant 0$, thus
\[
\al_\varepsilon \lambda_\varepsilon \geqslant \int_\Om \exp (\al_\varepsilon u_\varepsilon ^2 ) dx -|\Om|.
\]
Dividing both sides by $\alpha_\varepsilon$ and sending $\varepsilon$ to zero, we obtain
\[
\liminf_{\varepsilon \to 0} \lam_\varepsilon \geqslant \frac1{2\pi} \Big[ \sup_{u\in \Hcirc (\Om), \|u\|_{1,\alpha} \leqslant 1} \int_\Om e^{2\pi u^2} dx -|\Om| \Big] >0
\]
thanks to Lemma \ref{lemLim=Sup} below. Thus we have \eqref{eq:liminf} as claimed in \eqref{eq:liminf}.
\end{proof}

Note that by the elementary inequality $t e^{t^2} \leqslant e + t^2 e^{t^2}$ for any $t \geqslant 0$, we conclude that $|\mu_\varepsilon | \leqslant e |\Om| + \lam_\varepsilon $. Hence, there is $c>0$ such that
\begin{equation}\label{eq:upbound}
\lam_\varepsilon ^{-1} |\mu_\varepsilon | \leqslant c
\end{equation}
for all $\varepsilon >0$.

%%%%%%%%%%%%
%%%%%%%%%%%%
%%%%%%%%%%%%
%%%%%%%%%%%%

\section{Asymptotic behavior of extremals for subcritical functionals}
\label{sec-Asymptotic}

In this section, we study the asymptotic behavior of functions $u_\varepsilon $ given in section \S\ref{sec-ExtremalFun}. Denote $c_\varepsilon = \max_{\o \Om} |u_\varepsilon |$. If $c_\varepsilon $ is bounded, then by applying standard elliptic theory to \eqref{eq:EL}, we see that $u_\varepsilon \to u^*$ in $C^2(\o \Om)$, which implies Theorems \ref{Main1} and \ref{Main2}. Hence, without loss of generality, we assume that
\begin{equation}\label{eq:assumption}
c_\varepsilon = u_\varepsilon (x_\varepsilon ) = \max_{\o \Om} |u_\varepsilon | \to \infty
\end{equation}
for some sequence of point $\{x_\varepsilon\}$ converging to some point $p \in \overline \Omega$. In the sequel, we do not distinguish a sequence and its subsequence. The reader can understand it from the context. 

First, an application of the Poincar\'e inequality implies that $\{u_\varepsilon \}_\varepsilon $ is bounded in $H^1(\Om)$. From this we can deduce that as $\varepsilon \to 0$
\begin{itemize}
 \item $u_\varepsilon \rightharpoonup u_0$ weakly in $H^1(\Om)$, 
 \item $u_\varepsilon \to u_0$ in $L^q(\Om)$ for any $1 \leqslant q < +\infty$ and
 \item $u_\varepsilon \to u_0$ a.e. in $\Om$.
\end{itemize} 
If $u_0\not \equiv0$, then there exist $r>1$ such that $\exp (\al_\varepsilon u_\varepsilon ^2 )$ is bounded in $L^r(\Om)$ provided $\varepsilon >0$ small enough. Applying standard elliptic theory to \eqref{eq:EL}, we get that $c_\varepsilon $ is bounded, which is impossible. Hence $u_0\equiv 0$. 

In the rest of the present section, we examine the blow-up sequence $\{u_\varepsilon\}$ as well as the blow-up ponit $p$. Our first property involves the blow-up point. 

\begin{claim}\label{claim1}
There holds $p\in \pa\Om$. 
\end{claim}

\begin{proof}[Proof of Claim \ref{claim1}]
Indeed, if otherwise, we can take $r>0$ sufficiently small such that $B_r(p) \subset \Om$. Considering the cut-off function $\chi \in C_0^\infty(B_r(p))$ such that $0\leqslant \chi\leqslant 1$ and $\chi =1$ in $B_{r/2}(p)$. Fix a small number $\delta > 0$, we have
\begin{align*}
\int_{\Om} |\na(\chi u_\varepsilon )|^2 dx &= \int_\Om |\chi\na u_\varepsilon + u_\varepsilon \na \phi|^2 dx \\
&\leqslant (1+ \de)\int_\Om |\na u_\varepsilon |^2 dx + \Big(1+\frac1\delta \Big) \int_\Om u_\varepsilon ^2 |\na \phi|^2 dx\\
&\leqslant (1+\de) + \Big[(1+\de)\alpha +C \Big(1+\frac1\delta \Big) \Big]\|u_\varepsilon \|_2^2,
\end{align*}
where $C= \sup |\na \chi|^2$. Thus for $\varepsilon $ small enough, we get that $\|\na (\chi u_\varepsilon )\|_2^2 \leqslant 1+ 2\delta < 3/2$ provided $\delta \in (0,1/4)$. Applying the Moser--Trudinger inequality \eqref{eq:MT}, we see that $\exp (\al_\varepsilon \chi^2u_\varepsilon )$ is bounded in $L^q(\Om)$ for some $q >1$, hence $\exp (\al_\varepsilon u_\varepsilon ^2 )$ is bounded in $L^q(B_{r/2}(p))$ for some $q >1$. Applying standard elliptic theory to \eqref{eq:EL}, we obtain the boundedness of $u_\varepsilon $ in $C^1(B_{r/2}(p))$. In particular, $c_\varepsilon $ is bounded, which contradicts to \eqref{eq:assumption}.
\end{proof}

Keep in mind that $u_\varepsilon \to 0$ strongly in any $L^q(\Omega)$ with $1<q<+\infty$. Next we want to show the following.

\begin{claim}\label{claim2}
As $\varepsilon \to 0$ there holds
\begin{equation}\label{eq:diract}
|\na u_\varepsilon |^2 dx \rightharpoonup \de_p
\end{equation}
in the sense of measure.
\end{claim}

\begin{proof}[Proof of Claim \ref{claim2}]
 Indeed, by the definition of $\| \cdot \|_{1, \alpha}$ we observe that $\|\na u_\varepsilon \|_2^2 = 1 + \al \|u_\varepsilon \|_2^2 \to 1$ as $\varepsilon \to 0$. If \eqref{eq:diract} does not hold, then there exist $r>0$ small enough and $\mu < 1$ such that 
\[
\lim_{\varepsilon \to 0} \int_{B_r(p)\cap \Om} |\na u_\varepsilon |^2 dx < \mu.
\]
Still let $\chi$ be a cut-off function as above, define
\[
\chi_\varepsilon =\chi u_\varepsilon -\frac1{|\Om|} \int_\Om \chi u_\varepsilon dx.
\] 
Then $\int_\Om \chi_\varepsilon dx =0$ and by the similar estimate as in the proof of Claim \ref{claim1} we have
\[
\int_\Om |\na \chi_\varepsilon |^2 dx \leqslant (1+ \de)\int_{B_r(p)} |\na u_\varepsilon |^2 dx + \Big( 1 + \frac 1\delta \Big) \int_\Om |u_\varepsilon |^2 |\na \chi|^2 dx,
\]
for any $\delta >0$. Since $u_\varepsilon \to 0$ in $L^2(\Om)$, $|\na \chi|$ is bounded, and $\mu \in (0,1)$, by fixing $\delta >0$ sufficient small, there exists some $\varepsilon _0 >0$ such that
\[
\int_\Om |\na \chi_\varepsilon |^2 dx < \frac{1+\mu}{2} < 1
\]
for all $\varepsilon < \varepsilon _0$. Thanks to $\alpha_\varepsilon = 2\pi - \varepsilon$, we apply the Moser--Trudinger inequality of Chang and Yang \eqref{eq:ChangYang} to obtain the boundedness of $\exp (\al_\varepsilon \chi_\varepsilon ^2 )$ in $L^s(B_{r/2}(p)\cap \Om)$ for some $s>1$. Note that
\[
\chi^2 u_\varepsilon ^2 \leqslant (1+t)\chi_\varepsilon ^2 + \frac{1+t}t \lt(\frac{1}{|\Om|} \int_\Om u_\varepsilon \chi dx\rt)^2
\]
for any $t >0$ and that $u_\varepsilon \to 0$ in $L^2(\Om)$. Therefore, by choosing $t >0$ small, we easily verify that $\exp (\al_\varepsilon u_\varepsilon ^2 )$ is bounded in $L^q(B_{r/2}(p)\cap \Om)$ for some $q >1$. Notice from \eqref{eq:EL} that $\pa u_\varepsilon /\pa \nu =0$ on $\pa \Om$. By boundary elliptic estimate, we obtain the boundedness of $u_\varepsilon $ near $p$, which contradicts to \eqref{eq:assumption}. Thus $|\na u_\varepsilon |^2 dx \rightharpoonup \de_p$ in the sense of measure.
\end{proof}

Denote $r_\varepsilon = \sqrt{\lam_\varepsilon} c_\varepsilon ^{-1} e^{- (\al_\varepsilon/2) c_\varepsilon ^2}$. Our next task is to estimate $r_\varepsilon$.

\begin{claim}\label{claim3}
As $\varepsilon \to 0$ there holds $r_\varepsilon \to 0$.
\end{claim}

\begin{proof}[Proof of Claim \ref{claim3}]
Indeed, for any $\be < 2\pi$ fixed we have $\al_\varepsilon - \be >0$ for any $\varepsilon >0$ sufficient small. In addition, by the definition of $\lambda_\varepsilon$ in \eqref{eq:EL} and $c_\varepsilon$ in \eqref{eq:assumption}, there holds
\[
r_\varepsilon ^2c_\varepsilon ^2 e^{\beta c_\varepsilon ^2} = \int_\Om e^{(\alpha_\varepsilon - \beta) (u_\varepsilon ^2 - c_\varepsilon ^2) } e^{\beta u_\varepsilon ^2} u_\varepsilon ^2 dx \leqslant \int_\Om e^{\beta u_\varepsilon ^2} u_\varepsilon ^2 dx \to 0
\]
by \eqref{eq:ChangYang}, the H\"older inequality, and the fact $u_\varepsilon \to 0$ in $L^q(\Om)$ for any $1 \leqslant q < +\infty$. From this we obtain the desired limit because $c_\varepsilon \to + \infty$.
\end{proof}

We continue studying the blow-up behavior of $u_\varepsilon $ near $p$. Following the argument in \cite{LYzero}, let us take an isothermal coordinate system $(\mathcal U, \phi)$ around the blow-up point $p$ such that:
\begin{itemize}
 \item $\phi(p) =0$, 
 \item $\phi :\mathcal U\cap \pa \Om \to \B_1 \cap \pa \R_+^2$, and 
 \item $\phi: \mathcal U\cap \Om \to \B_1\cap \R^2_+$
\end{itemize}
where $\R^2_+ =\{(y_1,y_2): y_2 >0\}$ and the symbol $\B$ is used to denote balls in $\Rset^2$. In such the coordinates, the original flat metric $g =dx_1^2 + dx_2^2$ has the representation $g = e^{2f(y)}(dy_1^2 + dy_2^2)$ with $f(0) =0$. We define a new function $\util _\varepsilon $ on $\B_1$ by
\begin{equation}\label{eq:UTitle}
\util _\varepsilon (y) = 
\begin{cases}
(u_\varepsilon \circ \phi^{-1}) (y_1,y_2) &\mbox{if $y_2\geqslant 0$,}\\
( u_\varepsilon \circ \phi^{-1}) (y_1,-y_2) &\mbox{if $y_2 < 0$.}
\end{cases}
\end{equation}
Since $\pa_\nu u_\varepsilon =0$ on $\pa \Om$, there holds $\pa_{y_2} \util _\varepsilon (y_1,0) =0$; hence $\util _\varepsilon \in C^1(\B_1)$. Denote $y_\varepsilon =\phi(x_\varepsilon )$ and $\mathcal U_\varepsilon =\{y\in \R^2: y_\varepsilon + r_\varepsilon y \in \B_1\}$. Since $y_\varepsilon \to 0$ and $r_\varepsilon \to 0$, the set $\mathcal U_\varepsilon \to \R^2$. We define two sequences of scaled functions $\psi_\varepsilon $ and $\vphi_\varepsilon $ on $\mathcal U_\varepsilon $ by
\begin{equation}\label{eq:ScaledFunctions}
\left\{
\begin{split}
\psi_\varepsilon (y) &= \frac{\util _\varepsilon (y_\varepsilon + r_\varepsilon y)}{c_\varepsilon },\\
 \vphi_\varepsilon (y) &=c_\varepsilon (\util _\varepsilon (y_\varepsilon + r_\varepsilon y) -c_\varepsilon ).
\end{split}
\right.
\end{equation}
A straightforward computation shows that $\psi_\varepsilon $ and $\vphi_\varepsilon $ satisfy the following equations
\[
\left\{
\begin{split}
-\Delta_y \psi_\varepsilon =& c_\varepsilon^{-2} \psi_\varepsilon e^{\alpha_\varepsilon c_\varepsilon ^2(\psi_\varepsilon ^2 -1)} + r_\varepsilon ^2 \alpha \psi_\varepsilon - c_\varepsilon^{-1} r_\varepsilon ^2 \frac{\mu_\varepsilon }{ \lam_\varepsilon },\\
-\Delta_y \vphi_\varepsilon =& \psi_\varepsilon e^{\al_\varepsilon \vphi_\varepsilon (1 + \psi_\varepsilon )} + c_\varepsilon r_\varepsilon ^2 \alpha \psi_\varepsilon -c_\varepsilon r_\varepsilon ^2 \frac{\mu_\varepsilon }{\lam_\varepsilon },
\end{split}
\right.
\]
on $\mathcal U_\varepsilon $. Since $|\psi_\varepsilon |\leqslant 1$, we know that as $\varepsilon \to 0$
\[
|\Delta_y \psi_\varepsilon |\leqslant c_\varepsilon ^{-2} + r_\varepsilon ^2 \al + c_\varepsilon^{-1} r_\varepsilon ^2 \frac{\mu_\varepsilon }{\lam_\varepsilon } \to 0
\]
uniformly in $\B_R(0)$ for a fixed $R>0$, here we use \eqref{eq:upbound}. Since $\psi_\varepsilon (0) =1$, by standard elliptic theory, we get $\psi_\varepsilon \to 1$ in $C^1(\o{\B_{R/2}(0)})$. Since $\vphi_\varepsilon \leqslant \vphi_\varepsilon (0) =0$, again using standard elliptic theory, we also get $\vphi_\varepsilon \to \vphi$ in $C^1(\o{\B_{R/4}(0)})$ for any $R >0$. Such a local convergence in $\Rset^2$ implies that $\vphi$ solves
\begin{equation}\label{eq:vphi}
\begin{cases}
-\Delta \vphi = e^{4\pi \vphi} &\mbox{in $\R^2$,}\\
\vphi \leqslant \vphi(0) =0,\\
\int_{\R^2} e^{4\pi \vphi} dx \leqslant 2.
\end{cases}
\end{equation}
%Moreover $$. 
%Indeed, for any $R >0$, we have
%\begin{align*}
%\int_{B_R(0)} e^{4\pi \vphi} dy &= \lim_{\varepsilon \to 0} \int_{B_R(0)} e^{\alpha_\varepsilon (1 + \psi_\varepsilon ) \vphi_\varepsilon }dy =\lim_{\varepsilon \to 0} \int_{B_R(0)} e^{\alpha_\varepsilon (\util _\varepsilon (y_\varepsilon + r_\varepsilon y)^2 -c_\varepsilon ^2)}dy\\
%&=\lim_{\varepsilon \to 0} r_\varepsilon ^{-2}\int_{B_{r_\varepsilon R}(y_\varepsilon )} e^{\alpha_\varepsilon (\util _\varepsilon (y)^2 -c_\varepsilon ^2)}dy \\
%&=\lim_{\varepsilon \to 0} \frac{c_\varepsilon ^2\int_{B_{r_\varepsilon R}(y_\varepsilon )} e^{\alpha_\varepsilon \util _\varepsilon (y)^2}dy}{\int_{\Om} u_\varepsilon ^2 e^{\alpha_\varepsilon ^2 u_\varepsilon ^2} dx}.
%\end{align*}
%Denote $ $
Using a well-known classification result of Chen and Li \cite{ChenLi}, we get that
\begin{equation}\label{eq:vphian}
\vphi(x) = -\frac1{2\pi} \log \Big(1 + \frac{\pi}2 |x|^2 \Big)
\end{equation}
in $\R^2$. In particular, there holds $\int_{\R^2} e^{4\pi \vphi} dx =2$. By writing $y_\varepsilon =((y_\varepsilon )_1, (y_\varepsilon )_2)$, we also get the following claim.

\begin{claim}\label{claim4}
As $\varepsilon \to 0$ there holds $r_\varepsilon ^{-1}(y_\varepsilon )_2 \to 0$.
\end{claim}
%\begin{equation}\label{eq:comp}
%\limsup_{\varepsilon \to 0} \frac{(y_\varepsilon )_2}{r_\varepsilon } = 0.
%\end{equation}

\begin{proof}[Proof of Claim \ref{claim4}]
Indeed, by way of contradiction, we have that 
$$\limsup_{\varepsilon \to 0} r_\varepsilon ^{-1}(y_\varepsilon )_2 = a >0$$ 
and that
\begin{align*}
\int_{\B_R(0)} e^{4\pi \vphi} dy &=\lim_{\varepsilon \to 0} \int_{\B_R(0)} e^{\al_\varepsilon (1+\psi_\varepsilon ) \vphi_\varepsilon } dy =\lim_{\varepsilon \to 0} \int_{\B_R(0)} e^{\alpha_\varepsilon \util _\varepsilon ^2(y_\varepsilon + r_\varepsilon y) -\al_\varepsilon c_\varepsilon ^2} dy\\
&=\lim_{\varepsilon \to 0} \frac{c_\varepsilon ^2}{\lam_\varepsilon }\int_{\B_{Rr_\varepsilon }(y_\varepsilon )} e^{\al_\varepsilon \util _\varepsilon ^2} dy \leqslant \lim_{\varepsilon \to 0} c_\varepsilon ^2 \frac{\int_{\B_{Rr_\varepsilon }(y_\varepsilon )} e^{\al_\varepsilon \util _\varepsilon ^2} dy}{\int_{\B_{Rr_\varepsilon }(y_\varepsilon )\cap \R^2_+} \util _\varepsilon ^2 e^{\al_\varepsilon \util _\varepsilon ^2} dy}\\
&=\lim_{\varepsilon \to 0} (1+ o_\varepsilon (R))\frac{\int_{\B_{Rr_\varepsilon }(y_\varepsilon )} e^{\al_\varepsilon \util _\varepsilon ^2} dy}{\int_{\B_{Rr_\varepsilon }(y_\varepsilon )\cap \R^2_+} e^{\al_\varepsilon \util _\varepsilon ^2} dy},
\end{align*}
since $\util _\varepsilon ^2 = c_\varepsilon ^2(1+ o_\varepsilon (R))$ on $\B_{Rr_\varepsilon }(y_\varepsilon )$. Making use of a change of variables, we get
\begin{align*}
\int_{\B_R(0)} e^{4\pi \vphi} dy &\leqslant \lim_{\varepsilon \to 0} (1+ o_\varepsilon (R)) \frac{\int_{\B_R(0)} e^{\al_\varepsilon (1+\psi_\varepsilon ) \vphi_\varepsilon } dy}{\int_{\B_R(0)\cap\{y_2 >-(y_\varepsilon )_2/r_\varepsilon \}} e^{\al_\varepsilon (1+\psi_\varepsilon ) \vphi_\varepsilon } dy}\\
&= \frac{\int_{\B_R(0)} e^{4\pi \vphi} dy}{\int_{\B_R(0)\cap\{y_2 \geqslant -a\}} e^{4\pi \vphi} dy}
\end{align*}
Letting $R\to\infty$ we get $\int_{\R^2} e^{4\pi \vphi} dy < 2$ which is impossible.
\end{proof}

For any $c>1$, define $u_\varepsilon ^c =\min\{c_\varepsilon /c, u_\varepsilon \}$. Then we have the following result.

\begin{lemma}\label{truncation}
For any $c>1$, there holds
\[
\lim_{\varepsilon \to 0} \int_\Om |\na u_\varepsilon ^c|^2 dx = \frac1c.
\]
\end{lemma}
\begin{proof}
We follows the arguments in \cite{Li2001}; see also \cite{Yang06}. Since $\pa_\nu u_\varepsilon =0$ on $\pa \Om$, using integration by parts we have
\begin{align*}
\int_\Om |\na u_\varepsilon ^c|^2 dx &= \int_\Om \na u_\varepsilon ^c \na u_\varepsilon dx = \int_\Om u_\varepsilon ^c (-\Delta u_\varepsilon ) dx\\
&=\frac1{\lam_\varepsilon } \int_\Om u_\varepsilon ^c u_\varepsilon e^{\al_\varepsilon u_\varepsilon ^2}dx + \al \int_\Om u_\varepsilon ^c u_\varepsilon dx -\frac{\mu_\varepsilon }{\lam_\varepsilon } \int_\Om u_\varepsilon ^c dx\\
&\geqslant \frac{c_\varepsilon }{c\lam_\varepsilon } \int_{\{u_\varepsilon > c_\varepsilon /c\}} u_\varepsilon e^{\al_\varepsilon u_\varepsilon ^2} dx + o_\varepsilon (1).
\end{align*}
For any $R >0$, since $\psi_\varepsilon \to 1$ in $C^1(\o{\B_R(0)})$, there holds $\phi^{-1}(\B_{Rr_\varepsilon }(y_\varepsilon )\cap \R^2_+) \subset \{u_\varepsilon > c_\varepsilon /c\}$ for $\varepsilon >0$ sufficient small. Thus
\begin{align*}
\int_\Om |\na u_\varepsilon ^c|^2 dx &\geqslant \frac{c_\varepsilon }{c\lam_\varepsilon } \int_{\phi^{-1}(\B_{Rr_\varepsilon }(y_\varepsilon )\cap \R^2_+)} u_\varepsilon e^{\al_\varepsilon u_\varepsilon ^2} dx + o_\varepsilon (1)\\
&=\frac{c_\varepsilon }{c\lam_\varepsilon } \int_{\B_{Rr_\varepsilon }(y_\varepsilon )\cap \R^2_+} \util _\varepsilon e^{\al_\varepsilon \util _\varepsilon ^2} dy + o_\varepsilon (1)\\
&=\frac1{c} \int_{\B_R(0)\cap\{y_2 > - (y_\varepsilon )_2 / r_\varepsilon \}}\psi_\varepsilon e^{\al_\varepsilon (1+ \psi_\varepsilon ) \vphi_\varepsilon } dy + o_\varepsilon (1).
\end{align*}
Whence
\[
\lim_{\varepsilon \to 0} \int_\Om |\na u_\varepsilon ^c|^2 dx \geqslant \frac1c \int_{\B_R(0)\cap \{y_2\geqslant 0\}} e^{4\pi \vphi} dy.
\]
Let $R\to \infty$ we get 
\[
\lim_{\varepsilon \to 0} \int_\Om |\na u_\varepsilon ^c|^2 dx \geqslant \frac 1c.
\] 
By the same way, we have
\[
\lim_{\varepsilon \to 0} \int_\Om \Big|\na\Big( u_\varepsilon -\frac{c_\varepsilon }c\Big)_+\Big|^2 dx \geqslant 1 -\frac1c.
\]
(Here we use the notation $f_+$ to denote the positive part of $f$.) Notice that
\[
\int_\Om |\na u_\varepsilon ^c|^2 dx + \int_\Om \Big|\na\Big(u_\varepsilon -\frac{c_\varepsilon }c\Big)_+\Big|^2 dx = \int_\Omega |\na u_\varepsilon |^2 dx = 1 + o_\varepsilon (1).
\]
From this we get the conclusion.
\end{proof}

\begin{lemma}\label{eq:rel}
There holds
\[
\limsup_{\varepsilon \to 0} \int_\Om e^{\al_\varepsilon u_\varepsilon ^2} dx \leqslant |\Om| + \limsup_{\varepsilon \to 0} \frac{\lam_\varepsilon }{c_\varepsilon ^2}.
\]
\end{lemma}

\begin{proof}
Fix $0< c< 1$ and define
\[
\left\{
\begin{split}
u_\varepsilon ^c &= \min\{u_\varepsilon , c_\varepsilon /c\},\\
v_\varepsilon ^c &=u_\varepsilon ^c -\frac1{|\Om|} \int_\Om u_\varepsilon ^c dx.
\end{split}
\right.
\] 
By Lemma \ref{truncation} and \eqref{eq:ChangYang}, there exists $p >1$ such that $\exp (\al_\varepsilon (v_\varepsilon ^c)^2)$ is bounded in $L^p(\Om)$ for $\varepsilon $ small. Since $u_\varepsilon \to 0$ in $L^r(\Om)$ for any $r< +\infty$, then there exists $q >1$ such that $\exp (\al_\varepsilon (u_\varepsilon ^c)^2)$ is bounded in $L^q(\Om)$ for $\varepsilon $ small enough. Thus
\[
\lim_{\varepsilon \to 0} \int_\Om e^{\al_\varepsilon (u_\varepsilon ^c)^2} dx = |\Om|.
\]
In the other hand, we have
\begin{align*}
\int_\Om e^{\al_\varepsilon u_\varepsilon ^2} &\leqslant \int_\Om e^{\al_\varepsilon (u_\varepsilon ^c)^2} dx + \int_{\{u_\varepsilon > c_\varepsilon /c\}} e^{\al_\varepsilon u_\varepsilon ^2} dx\\
&\leqslant \int_\Om e^{\al_\varepsilon (u_\varepsilon ^c)^2} dx + \frac{c^2}{c_\varepsilon ^2} \int_\Om u_\varepsilon ^2 e^{\al_\varepsilon (u_\varepsilon ^c)^2} dx\\
&=\int_\Om e^{\al_\varepsilon (u_\varepsilon ^c)^2} dx + c^2\frac{\lam_\varepsilon }{c_\varepsilon ^2}.
\end{align*}
Letting $\varepsilon \to 0$ and then letting $c\to 1$ we get the conclusion.
\end{proof}

\begin{lemma}\label{lemLim=Sup}
There holds
\[
\limsup_{\varepsilon \to 0} \int_\Om e^{\al_\varepsilon u_\varepsilon ^2} dx = \sup_{u\in \Hcirc (\Omega), \|u\|_{1,\alpha} =1} \int_{\Omega} e^{2\pi u^2} dx.
\]
\end{lemma}

\begin{proof}
This is elementary. Indeed, using the definition of $C_\varepsilon$ and \eqref{eq:sub} we know that $\int_\Om \exp \big( (2\pi-\varepsilon ) u_\varepsilon ^2 \big) dx$ is monotone increasing with respect to $\varepsilon>0$. Hence the limit $\limsup_{\varepsilon \to 0}\int_\Om \exp \big( (2\pi-\varepsilon ) u_\varepsilon ^2 \big) dx$ exists, however, it could be infinity. To conclude the lemma, we first observe that for arbitrary $\varepsilon$, the function $u_\varepsilon \in \Hcirc (\Om)$ with $\|u_\varepsilon \|_{1,\alpha} \leqslant 1$. Therefore,
\[
\limsup_{\varepsilon \to 0} \int_\Om e^{\al_\varepsilon u_\varepsilon ^2} dx  \leqslant \sup_{u\in \Hcirc (\Om), \|u\|_{1,\alpha} \leqslant 1} \int_\Om e^{2\pi u^2} dx.
\]
Conversely, for any function $u\in \Hcirc (\Om)$ satisfying $\|u\|_{1,\al} =1$, by the Fatou lemma, we have
\[
\int_\Om e^{2\pi u^2} dx \leqslant \liminf_{\varepsilon \to 0} \int_\Om e^{\al_\varepsilon u^2} dx \leqslant \liminf_{\varepsilon \to 0} \int_\Om e^{\al_\varepsilon u_\varepsilon ^2} dx .
\]
Taking supremum all over such functions $u$, we get 
\[
\sup_{u\in \Hcirc (\Om), \|u\|_{1,\alpha} \leqslant 1} \int_\Om e^{2\pi u^2} dx \leqslant \liminf_{\varepsilon \to 0} \int_\Om e^{\al_\varepsilon u_\varepsilon ^2} dx.
\]
Thus, we have proved that 
\[
\lim_{\varepsilon \to 0} \int_\Om e^{\al_\varepsilon u_\varepsilon ^2} dx = \sup_{u\in \Hcirc (\Om), \|u\|_{1,\alpha} \leqslant 1} \int_\Om e^{2\pi u^2} dx.
\]
as claimed.
\end{proof}

Although we do have the strong convergence $u_\varepsilon \to 0$ in $L^p (\Omega)$ for any $1 \leqslant p < +\infty$ and the convergence in measure established in Claim \ref{claim2}, it is not clear how $u_\varepsilon$ converges; in fact, we can say more about $u_\varepsilon$. In the final part of our blow-up analysis, we provide an asymptotic behavior of $u_\varepsilon $ away from the blow-up point $p$.

\begin{proposition}\label{Green}
We have $c_\varepsilon u_\varepsilon \rightharpoonup \G$ weakly in $W^{1,q}(\Om)$ for any $1< q < 2$, where $\G\in C^\infty(\o{\Om}\setminus\{p\})$ is a Green function satisfying the following equation
\begin{equation}\label{eq:Green}
\left\{
\begin{split}
-\Delta \G &= \de_p + \alpha \G -|\Om|^{-1} \qquad \mbox{ in } \Om,\\
\pa_\nu \G &=0 \qquad \mbox{ on } \pa \Om \setminus \{p\},\\
\int_\Om \G dx &= 0,
\end{split}
\right.
\end{equation}
where $\de_p$ is the Dirac measure at $p$. Moreover, there holds $c_\varepsilon u_\varepsilon \to \G$ in $C^\infty_{\rm loc}(\Omega)$.
\end{proposition}

\begin{proof}
The proof of this proposition is similar to the proof of Lemma 4.9 in \cite{Yang07} with only slight modification is needed; therefore we omit the details.
\end{proof}

For future benefit, it is worth noticing that the Green function $\G$ appearing in \eqref{eq:Green} above takes the form
\begin{equation}\label{eq:GreenExpansion}
\G(x) = -\frac1\pi \log (|x-p|) + A_p + \beta(x),
\end{equation}
where $A_p$ is constant, $\beta \in C^1(\o{\Om})$, and $\beta(x) =O(|x-p|)$.

%%%%%%%%%%%%
%%%%%%%%%%%%
%%%%%%%%%%%%
%%%%%%%%%%%%

\section{Capacity estimates}
\label{sec-Capacity}

In this section we use capacity techniques to calculate $\limsup_{\varepsilon \to 0} \lam_\varepsilon  c_\varepsilon ^{-2}$; hence by Lemma \ref{eq:rel} we have an upper bound of
\[
\sup_{u\in \Hcirc (\Om), \|u\|_{1,\alpha} \leqslant 1} \int_{\Omega} e^{2\pi u^2} dx,
\]
under the blow-up assumption, that is $c_\varepsilon \to \infty$ as $\varepsilon \to 0$. The main result of this section is given in Proposition \ref{eq:upperbound} below. We note here that the technique of using capacity estimate applied to this kind of problems was discovered by Li \cite{Li2001} in dealing with the Moser--Trudinger inequality.

Take an isothermal coordinate $(\mathcal U,\phi)$ around the blow-up point $p$ as above. We denote $y_\varepsilon = \phi(x_\varepsilon )$ and define the function $\util _\varepsilon $ as in \eqref{eq:UTitle}. In such a coordinate system, the original flat metric $g= dx_1^2 + dx_2^2$ has the form 
$$g =e^{2f(y)}(dy_1^2 + dy_2^2)$$ 
with $f(0) =0$. Then by a simple change of variables, we have
\[
|\na u_\varepsilon |^2 dx_1 dx_2 = |\na \util _\varepsilon |^2 dy_1 dy_2.
\]
Fix $0 < \delta < 1/2$ and $R >0$, thanks to Claim \ref{claim3}, for sufficiently small $\varepsilon$ we can define
\[
T_\varepsilon (a,b) = \big\{ u \in H^1(\B_\de(y_\varepsilon )\setminus \B_{r_\varepsilon R}(y_\varepsilon )): u = a \text{ on } \pa \B_\de(y_\varepsilon ), \, u = b \text{ on } \pa \B_{R r_\varepsilon }(y_\varepsilon ) \big\}.
\]
Denote
\[
\left\{
\begin{split}
s_\varepsilon =& \sup_{\pa \B_\de(y_\varepsilon )} \util _\varepsilon,\\
i_\varepsilon =& \inf_{\pa \B_{Rr_\varepsilon }(y_\varepsilon )} \util _\varepsilon .
\end{split}
\right.
\] 
Recall the definition of $\varphi_\varepsilon$ in \eqref{eq:ScaledFunctions} and the convergence $\vphi_\varepsilon \to \vphi$ in $C^1(\o{\B_{R/4}(0)})$ for any $R >0$ where $\vphi$ is given in \eqref{eq:vphian}. By this convergence result, we know that
\[
\util _\varepsilon (y_\varepsilon + r_\varepsilon \cdot) \to c_\varepsilon + \frac 1{c_\varepsilon} \Big(  -\frac1{2\pi} \log \Big(1 + \frac{\pi}2 | \cdot |^2 \Big) \Big)
\]
in $C^1(\o{\B_{R/4}(0)})$. Therefore, on the boundary $\pa \B_{R r_\varepsilon }(y_\varepsilon )$ shrinking to zero we obtain
\begin{equation}\label{eq:ExpansionI}
i_\varepsilon = c_\varepsilon + \frac1{c_\varepsilon } \Big(-\frac1{2\pi} \log \Big(1 + \frac{\pi R^2}2 \Big) + o_\varepsilon (R)+ o_\varepsilon (1) \Big)
\end{equation}
while on the fixed boundary $\pa \B_\de(y_\varepsilon )$ far from zero we deduce from Proposition \ref{Green} that
\begin{equation}\label{eq:ExpansionS}
s_\varepsilon = \frac1{c_\varepsilon } \Big(-\frac1\pi \log \delta + A_p + o_\de(1) + o_\varepsilon (1)\Big),
\end{equation}
where the errors $o_\varepsilon (1)\to 0$ as $\varepsilon \to 0$, $o_\de(1)\to 0$ as $\delta \to 0$, and $o_\varepsilon (R) \to 0$ as $\varepsilon \to 0$ for any fixed $R$. An immediate consequence of these expansions is that for $\varepsilon $ sufficient small, we get $s_\varepsilon \leqslant i_\varepsilon $. A simple variational technique implies that the value
\[
\inf_{\util \in T_\varepsilon (s_\varepsilon ,i_\varepsilon )} \int_{\B_\de(y_\varepsilon )\setminus \B_{r_\varepsilon R}(y_\varepsilon )} |\na \util|^2 dy
\]
is attained by a function $\htil \in H^1( \B_\de(y_\varepsilon ) \setminus \o{\B_{Rr_\varepsilon }(y_\varepsilon )})$ satisfying
\[
\left\{
\begin{split}
\Delta \htil  & = 0 \qquad \mbox{ in } \B_\de(y_\varepsilon ) \setminus \o{\B_{Rr_\varepsilon }(y_\varepsilon )},\\
\htil  &= s_\varepsilon \qquad \mbox{ on } \pa \B_\de(y_\varepsilon ),\\
\htil  &= i_\varepsilon \qquad \mbox{ on } \pa \B_{Rr_\varepsilon }(y_\varepsilon ).
\end{split}
\right.
\]
In fact, it is not hard to verify that the function $h$ is given as follows
\[
\htil  (y) = \frac{s_\varepsilon (\log |y-y_\varepsilon | - \log (R r_\varepsilon )) + i_\varepsilon (\log \delta -\log |y-y_\varepsilon |)}{\log \delta -\log (R r_\varepsilon )},
\]
and hence by a direct calculation
\begin{equation}\label{eq:energyh}
\int_{\B_\de(y_\varepsilon )\setminus \B_{Rr_\varepsilon }(y_\varepsilon )}|\na \htil |^2 dy = \frac{2\pi(s_\varepsilon -i_\varepsilon )^2}{\log \delta -\log (R r_\varepsilon )}.
\end{equation}
We now estimate the left hand side of \eqref{eq:energyh} from above. Set 
$$\util _\varepsilon ' = \max \big\{s_\varepsilon , \min\{\util _\varepsilon , i_\varepsilon \} \big\}.
$$
Clearly $\util _\varepsilon '\in T_\varepsilon (s_\varepsilon ,i_\varepsilon )$; therefore we have
\begin{equation}\label{eq:energyh-1}
\begin{split}
\int_{\B_\de(y_\varepsilon )\setminus \B_{r_\varepsilon R}(y_\varepsilon )} |\na \htil |^2 dy &\leqslant \int_{\B_\de(y_\varepsilon )\setminus \B_{r_\varepsilon R}(y_\varepsilon )} |\na \util _\varepsilon '|^2 dy\\
&\leqslant \int_{\B_\de(y_\varepsilon )\setminus \B_{r_\varepsilon R}(y_\varepsilon )} |\na \util _\varepsilon |^2 dy\\
&=\int_{\B_\de(y_\varepsilon )} |\na \util _\varepsilon |^2 dy - \int_{\B_{Rr_\varepsilon }(y_\varepsilon )} |\na \util _\varepsilon |^2 dy.
\end{split}
\end{equation}
In view of Proposition \ref{Green}, we can estimate the first term on the far right hand side of \eqref{eq:energyh-1} as follows
\begin{equation}\label{eq:energyh-2}
\begin{split}
\int_{\B_\de(y_\varepsilon )} |\na \util _\varepsilon |^2 dy &= \int_{\B_\de(y_\varepsilon ) \cap \R^2_+} |\na \util _\varepsilon |^2 dy + \int_{\B_\de(y_\varepsilon ) \cap \R^2_-} |\na \util _\varepsilon |^2 dy\\
&\leqslant 2\int_{\B_\de(y_\varepsilon ) \cap \R^2_+} |\na \util _\varepsilon |^2 dy\\
&= 2\int_{\phi^{-1}(\B_\de(y_\varepsilon ))} |\na u_\varepsilon |^2 dx\\
&=2 + 2\al \|u_\varepsilon \|_2^2 -2\int_{\Om \setminus \phi^{-1}(\B_\de(y_\varepsilon ))} |\na u_\varepsilon |^2 dx\\
&=2 + \frac{2}{c_\varepsilon ^2} \Big(\alpha \|\G\|_2^2 -\int_{\Om \setminus \phi^{-1}(\B_\de(0))}|\na \G|^2 dx + o_\varepsilon (1) + o_\varepsilon (\de)\Big),
\end{split}
\end{equation}
where $o_\varepsilon (1) \to 0$ as $\varepsilon \to 0$ and $o_\varepsilon (\de) \to 0$ as $\varepsilon \to 0$ and $\de$ is fixed. Using \eqref{eq:Green}, \eqref{eq:GreenExpansion}, and integration by parts, we get
\begin{equation}\label{eq:energyh-3}
\int_{\Om \setminus \phi^{-1}(\B_\de(0))}|\na \G|^2 dx = -\frac1\pi \log \delta +A_p+ \al \|\G\|_2^2 + o_\varepsilon (\de) + o_\varepsilon (1) + o_\de(1).
\end{equation}
Recall that the scaled function $\vphi_\varepsilon \to \vphi$ in $C^1_{\rm loc}(\R^2)$ with $\vphi$ a standard solution given by \eqref{eq:vphian}. Whence
\begin{equation}\label{eq:energyh-3}
\begin{split}
\int_{\B_{Rr_\varepsilon }(y_\varepsilon )} |\na \util _\varepsilon |^2 dy &= \frac1{c_\varepsilon ^2} \int_{\B_R(0)} |\na \vphi_\varepsilon |^2 dy\\
&=\frac{1}{c_\varepsilon ^2} \Big( \int_{\B_R(0)} |\na \vphi|^2 dy + o_\varepsilon (R) \Big)\\
&= \frac1{c_\varepsilon ^2} \Big( \frac1{\pi} \log \lt(1 + \frac\pi 2 R^2\rt) -\frac1{\pi} +o_R(1) + o_\varepsilon (R) \Big),
\end{split}
\end{equation}
thanks to \eqref{eq:ScaledFunctions}. Consequently, combining \eqref{eq:energyh-1}--\eqref{eq:energyh-3} gives
\begin{equation}\label{eq:energyh-Upper}
\begin{split}
\int_{\B_\de(y_\varepsilon )\setminus \B_{r_\varepsilon R}(y_\varepsilon )} |\na \htil |^2 dy \leqslant 2+ \frac{2}{c_\varepsilon ^2} 
\Bigg(
\begin{split}
&\frac1\pi \log \delta -A_p -\frac1{2\pi}\log \lt(1 + \frac\pi 2 R^2\rt) + \frac1{2\pi}\\
&+ o_\varepsilon (\de) + o_\varepsilon (1) + o_\de(1)+ o_R(1) + o_\varepsilon (R)
\end{split}
\Bigg).
\end{split}
\end{equation}
We now go back to \eqref{eq:energyh} to estimate the right hand side of \eqref{eq:energyh}. From the definition of $r_\varepsilon $ right before Claim \ref{claim4}, we get that
\[
\frac{\log \delta -\log (Rr_\varepsilon )}{2\pi} = \frac{\log \delta -\log R}{2\pi}-\frac1{4\pi} \log \frac{\lam_\varepsilon }{c_\varepsilon ^2} + \frac{\al_\varepsilon c_\varepsilon ^2}{4\pi}.
\]
From the expression of $i_\varepsilon $ in \eqref{eq:ExpansionI} and $s_\varepsilon $ in \eqref{eq:ExpansionS}, we have
\[
(i_\varepsilon -s_\varepsilon )^2 = c_\varepsilon ^2 + \frac2\pi \log \delta -2A_p -\frac1\pi \log \lt(1+ \frac\pi 2 R^2 \rt) + o_\varepsilon (R) + o_\varepsilon (1) + o_\de(1).
\]
Thus, we obtain from the last two estimates together with \eqref{eq:energyh} and \eqref{eq:energyh-Upper} the following
\[
\begin{split}
c_\varepsilon ^2 + \frac2\pi \log \de& -2A_p -\frac1\pi \log \lt(1+ \frac\pi 2 R^2 \rt) + o_\varepsilon (R) + o_\varepsilon (1) + o_\de(1)\\
\leqslant &\lt(\frac{\log \delta -\log R}{2\pi}-\frac1{4\pi} \log \frac{\lam_\varepsilon }{c_\varepsilon ^2} + \frac{\al_\varepsilon c_\varepsilon ^2}{4\pi}\rt) \\
&\times \Bigg[2+ \frac{2}{c_\varepsilon ^2} 
\Bigg(
\begin{split}
& \frac1\pi \log \delta -A_p -\frac1{2\pi}\log \lt(1 + \frac\pi 2 R^2\rt) + \frac1{2\pi}\\
&+ o_\varepsilon (\de) + o_\varepsilon (1) + o_\de(1)+ o_R(1) + o_\varepsilon (R)
\end{split}
\Bigg)
\Bigg]\\
=& \frac{\log \delta -\log R}{\pi}-\frac1{2\pi}(1+o_\varepsilon (1) + o_\varepsilon (\de) + o_\varepsilon (R)) \log \frac{\lam_\varepsilon }{c_\varepsilon ^2}\\
& + \frac{\al_\varepsilon c_\varepsilon ^2}{2\pi} + \frac{\al_\varepsilon }{2\pi} \frac1\pi \log \delta -\frac{\al_\varepsilon }{2\pi} A_p -\frac{\al_\varepsilon }{2\pi} \frac1{2\pi}\log \lt(1 + \frac\pi 2 R^2\rt) + \frac{\al_\varepsilon }{4\pi^2} \\
&+ o_\varepsilon (\de) + o_\varepsilon (1) + o_\de(1)+ o_R(1) + o_\varepsilon (R).
\end{split}
\]
Thus, after a tedious computation, we arrive at
\begin{align*}
\frac1{2\pi}(1+o_\varepsilon (1) +& o_\varepsilon (\de) + o_\varepsilon (R)) \log \frac{\lam_\varepsilon }{c_\varepsilon ^2} \\
&\leqslant \frac1{2\pi} + \frac1{2\pi}\log \frac\pi2+ A_p+ o_\varepsilon (\de) + o_\varepsilon (1) + o_\de(1)+ o_R(1) + o_\varepsilon (R).
\end{align*}
We now let $\varepsilon \to 0$ and then let $\de\to 0$ and $R\to \infty$ to get from the preceding estimate the following
\[
\limsup_{\varepsilon \to 0} \Big( \frac{\lam_\varepsilon }{c_\varepsilon ^2} \Big) \leqslant \frac\pi 2 e^{1 + 2\pi A_p}.
\]
Combining the preceding estimate, Lemma \ref{eq:rel}, and Lemma \ref{lemLim=Sup} in the previous section, we obtain the following key estimate.

\begin{proposition}\label{eq:upperbound}
As $\varepsilon \to 0$, if $c_\varepsilon \to \infty$, then the inequality
\[
\sup_{u\in \Hcirc (\Om), \|u\|_{1,\alpha} \leqslant 1} \int_{\Omega} e^{2\pi u^2} dx \leqslant |\Om| + \frac\pi 2 e^{1 + 2\pi A_p}
\]
holds.
\end{proposition}

%%%%%%%%%%%%
%%%%%%%%%%%%
%%%%%%%%%%%%
%%%%%%%%%%%%

\section{Proofs of main theorems}
\label{sec-Proofs}

\subsection{Proofs of Theorems \ref{Main1} and \ref{Main2}}

This part is devoted to proofs of Theorems \ref{Main1} and \ref{Main2}. First we prove Theorem \ref{Main1}. If $c_\varepsilon $ is bounded, by applying standard elliptic theory to \eqref{eq:EL}, we see that $u_\varepsilon \to u^*$ in $C^2(\o \Om)$ which implies Theorems \ref{Main1}. If $c_\varepsilon \to \infty$ as $\varepsilon \to 0$, then Theorem \ref{Main1} follows from Proposition \ref{eq:upperbound}.

Next we prove Theorem \ref{Main2}. To this end, our aim is to construct a sequence $\phi_\varepsilon \in \Hcirc (\Om)$ such that $\|\na \phi_\varepsilon \|_{1,\al} =1$ and
\begin{equation}\label{eq:suff}
\int_\Om e^{2\pi \phi_\varepsilon ^2} dx > |\Om| + \frac\pi 2 e^{1 +2\pi A_p},
\end{equation}
for $\ep >0$ small enough. 

If, for a moment, this construction is possible, then in view of Proposition \ref{eq:upperbound} above we get the boundedness of $c_\varepsilon$. From this by considering the maximizing sequence $\{u_\varepsilon\}$ for $\{c_\varepsilon\}$ and applying standard elliptic theory to \eqref{eq:EL}, we see that $u_\varepsilon \to u^*$ in $C^2(\o \Om)$ for some function $u^*$. From this, it is routine to realize that $u^*$ is the optimal function we are looking form. The proof of Theorems \ref{Main2} then follows. Thus, in the rest of this section, we aim to construct a sequence $\{ \phi_\varepsilon \} \subset H^1(\Om)$ having all properties mentioned earlier. For clarity, we divide our construction into several steps.

\medskip\noindent\textbf{Step 1}. Let $p \in \partial \Omega$ be the blow-up point. Again we take an isothermal coordinates around $p$ represented by $\phi$. Starting from a sufficiently small $\varepsilon \in (0,1)$ chosen in such a way that we can identify, via $\phi$, $\Omega \cap B_{2R\varepsilon}(p)$ as a half-ball in $\Rset_+^2$ where $R = -\log \varepsilon $. (Keep in mind that our choice for $R$ guarantees that $R \nearrow +\infty$ and that $R\varepsilon \searrow 0$ as $\varepsilon \searrow 0$.) We consider two sequences of functions $\{w_\varepsilon \}_\varepsilon$ and $\{\phi_\varepsilon \}_\varepsilon$ defined by
\[
w_\varepsilon = \widetilde w_\varepsilon \circ \phi , \quad \phi_\varepsilon = w_\varepsilon - |\Om|^{-1}\int_\Om w_\varepsilon dx,
\]
where $\widetilde w_\varepsilon$ is a radially symmetric function centered at $\phi(p) \equiv 0$ given by
\[
\widetilde w_\varepsilon (r) =
\left\{
\begin{split}
c + \frac1c\lt(-\frac1{2\pi} \log \Big(1 + \frac\pi2 \frac{r^2}{\varepsilon ^2}\Big) + A\rt) \quad & \mbox{ if } 0 < r < R\varepsilon,\\
\frac{\G -\eta \beta}c \quad &\mbox{ if } R\varepsilon \leqslant r < 2R\varepsilon,\\
\frac \G c \quad & \mbox{ if } r \geqslant 2R\varepsilon,
\end{split}
\right.
\]
where $\eta$ is cut-off function in $B_{2R\varepsilon }(p)$ satisfying $\eta \equiv 1$ in $B_{R\varepsilon }(p)$ and $\|\na \eta\|_\infty = O((R\varepsilon )^{-1})$, and $c,A$ are constants to be determined later. (Here, in order to avoid introducing further notations, the Green function $\G$ is understood both in the orginial coordinates with center at $p$ or after making use of the isothermal coordinates with center at zero.) In order for $w_\varepsilon$ to belong to $H^1(\Om)$, we choose $A$ in such a way that $w_\varepsilon$ is continuous across $\partial B_{R\varepsilon}(p)$. This forces
\[\begin{split}
c + \frac1c\lt(-\frac1{2\pi} \log \lt(1 + \frac\pi2 R^2\rt) + A \rt) = & \lim_{r\nearrow R\varepsilon } w_\varepsilon (r) \\
= &\lim_{r\searrow R\varepsilon } w_\varepsilon (r) = \frac1c\lt(-\frac1\pi \log (R\varepsilon ) + A_p\rt),
\end{split}\]
which gives
\begin{align}\label{eq:A}
A &=-c^2 + \frac1{2\pi} \log \lt(1 + \frac\pi2 R^2\rt) -\frac1\pi \log (R\varepsilon ) + A_p\notag\\
&= -c^2 + \frac1{2\pi} \log\frac\pi2 -\frac1\pi \log\ep + A_p + O(R^{-2}).
\end{align}
From this we obtain 
\begin{equation}\label{eq:WInInnerBall}
c w_\varepsilon (r) \big|_{B_{ R\varepsilon}(p) \setminus\{p\} } =-\frac1{2\pi} \log \lt(1 + \frac\pi2 \frac{r^2}{\varepsilon ^2}\rt) + \frac1{2\pi} \log \lt(1 + \frac\pi2 R^2\rt) -\frac1\pi \log (R\varepsilon ) + A_p .
\end{equation}
Clearly $\int_\Omega \phi_\varepsilon dx =0$; hence $\phi_\varepsilon \in \Hcirc (\Omega)$. In the next step, we carefully select $c$ in such a way that $\|\na \phi_\varepsilon \|_{1,\al} =1$. Then in the last step, we verify \eqref{eq:suff}.

\medskip\noindent\textbf{Step 2}. In this step, to determine $c$, we first compute the Dirichlet integral $\int_\Om |\na w_\varepsilon |^2 dx$. Thanks to \eqref{eq:apdIntegralNablaPhi}, we have
%This integral can be computed by splitting $\int_\Omega = \int_{\Om \cap B_{R\varepsilon }(p)} + \int_{\Om \setminus B_{R\varepsilon }(p)} $. On the region $\Om \cap B_{R\varepsilon }(p)$, we use Appendix \ref{apd-IntegralNablaPhi-Inner} to get
%\[\int_{\Om \cap B_{R\varepsilon }(p)} |\na w_\varepsilon |^2 dx = \frac1{2\pi c^2} \lt(2\log R + \log \frac\pi2 -1 +O(R^{-2})\rt).\]
%On $\Om\setminus B_{Rr_\varepsilon }(p)$, we use Appendix \ref{apd-IntegralNablaPhi-Outter} to get
%\begin{align*}
%\int_{\Om \setminus B_{R\varepsilon }(p)} |\na w_\varepsilon |^2 dx =& \frac 1{c^2} \Big( \alpha %\int_{\Om\setminus B_{R\varepsilon }(p)} \G^2 dx -\frac {\log (R\varepsilon )}{\pi} + A_p  \Big)\\
%& + O(R\varepsilon \log (R\varepsilon )) + O(R\varepsilon) .
%\end{align*}
%All terms on the right hand side of the preceding equation still need an extra work. 
%Combining all computations above gives
\begin{equation}\label{eq:IntegraDeltaW}
\int_\Om |\na w_\varepsilon |^2 dx = \frac1{c^2} 
\left(\alpha\int_{\Om } \G^2 dx -\frac {\log \varepsilon}{\pi} + \frac1{2\pi}\log \frac\pi2 +A_p -\frac1{2\pi} + O(R^{-2})\right).
\end{equation}
%Note that to obtain the preceding equation, we have used the fact that $R\varepsilon \log (R\varepsilon )= O(R^{-2})$ to discard the term $O(R\varepsilon \log (R\varepsilon )) $.
We next compute $\int_\Om w_\varepsilon dx$ and $\int_\Om w_\varepsilon ^2 dx$. Using the main estimate in \eqref{eq:apdIntegralPhi}, we have
\[
\begin{split}
\int_\Om w_\varepsilon dx=  \frac1c O((R\varepsilon )^2 \log (R\varepsilon) ) .
\end{split}
\]
In particular, 
\begin{equation}\label{eq:meanwep}
\Big( \int_\Om w_\varepsilon dx \Big)^2= \frac1{c^2}O(R^{-2}).
\end{equation}
Similarly, using the main estimate \eqref{eq:apdIntegralPhi^2} in Appendix \ref{apd-IntegralPhi^2} we have
\begin{equation}\label{eq:IntegraW^2}
\begin{split}
\int_{\Om} w_\varepsilon ^2 dx =& \frac 1{c^2} \left(\int_{\Om } \G^2 dx + O((R\varepsilon )^2(\log (R\varepsilon ))^2)\right).
\end{split}
\end{equation}
Thus, combining \eqref{eq:IntegraDeltaW}, \eqref{eq:meanwep}, and \eqref{eq:IntegraW^2} we get
\[\begin{split}
\|\phi_\varepsilon \|_{1,\alpha}^2 =&\int_\Om |\na w_\varepsilon |^2 dx  -\alpha  \int_\Om w_\varepsilon ^2 dx + \alpha \frac1{|\Omega|} \Big( \int_\Om w_\varepsilon  dx \Big)^2\\
= &\frac1{c^2} \Big( -\frac {\log \varepsilon}\pi + \frac1{2\pi}\log \frac\pi2+ A_p  -\frac1{2\pi} + O(R\varepsilon \log (R\varepsilon )) + O(R^{-2})          \Big)\\
= &\frac1{c^2} \Big( -\frac {\log \varepsilon}\pi + \frac1{2\pi}\log \frac\pi2+ A_p  -\frac1{2\pi} + O\Big( \frac1{(\log \varepsilon )^2} \Big)\Big),
\end{split}\]
here we have already used $R = -\log \varepsilon $. Therefore, for $\varepsilon $ sufficient small, we can choose $c$ in such a way that $\|\phi_\varepsilon \|_{1,\al} =1$. Indeed, a direct computation leads us to
\begin{equation}\label{eq:cvalue}
c^2 = -\frac{\log \varepsilon}{\pi} + A_p + \frac1{2\pi}\log \frac\pi2 -\frac1{2\pi} + O\lt(\frac1{(\log \varepsilon )^2}\rt).
\end{equation}
In particular, it follows from \eqref{eq:A} that
\begin{equation}\label{eq:NewA}
A = \frac1{2\pi} + O(R\varepsilon \log (R\varepsilon )) + O(R^{-2}) .
\end{equation}
By now, we know that $\phi_\varepsilon \in \Hcirc (\Omega)$ with $\|\na \phi_\varepsilon \|_{1,\al} =1$. Then in the last step, we shall prove that \eqref{eq:suff} actually holds provided $\varepsilon$ is small.

\medskip\noindent\textbf{Step 3}. We next compute $\int_\Om e^{2\pi \phi_\varepsilon ^2} dx$. On the region $\Om\setminus B_{R\varepsilon }(p)$ we apply the elementary inequality $e^x \geqslant 1+x$ to get
\begin{align*}
\int_{\Om\setminus B_{R\varepsilon }(p)} e^{2\pi \phi_\varepsilon ^2} dx &\geqslant \int_{\Om\setminus B_{R\varepsilon }(p)}(1 + 2\pi \phi_\varepsilon ^2) dx\\
&= |\Om\setminus B_{R\varepsilon }| + \frac{2\pi}{c^2}\lt(\int_{\Om\setminus B_{R\varepsilon }(p)} \G ^2 dx + O(R^{-2})\rt)\\
&=|\Om| + \frac{2\pi}{c^2}\lt(\int_{\Om} \G ^2 dx + O(R^{-2})\rt) + O((R\varepsilon)^2)\\
&=|\Om| + \frac{2\pi}{c^2}\lt(\int_{\Om} \G ^2 dx + O(R^{-2})\rt),
\end{align*}
since $c^2 = O(R)$ by \eqref{eq:cvalue}. On $\Om \cap B_{R\varepsilon }(p)$, we use the formula for $w_\varepsilon$ to obtain
\begin{align*}
\phi_\varepsilon ^2 = & w_\varepsilon^2 + \Big(|\Omega|^{-1} \int_\Omega w_\varepsilon dx \Big)^2 - \frac 2{|\Omega|} w_\varepsilon \int_\Omega w_\varepsilon dx\\
= & c^2 + 2\Big(-\frac1{2\pi} \log \Big(1 + \frac\pi2 \frac{r^2}{\varepsilon ^2} \Big) + A\Big) \\
& + \Big(-\frac1{2\pi} \log \Big(1 + \frac\pi2 \frac{r^2}{\varepsilon ^2} \Big) + A\Big)^2 + \Big(|\Omega|^{-1} \int_\Omega w_\varepsilon dx \Big)^2 - \frac 2{|\Omega|} w_\varepsilon \int_\Omega w_\varepsilon dx\\
\geqslant & c^2 + 2\Big(-\frac1{2\pi} \log \Big(1 + \frac\pi2 \frac{r^2}{\varepsilon ^2} \Big) + A\Big) - \frac{2w_\varepsilon}{|\Omega|} \int_\Omega w_\varepsilon dx \\
= & c^2 + 2\Big(-\frac1{2\pi} \log \Big(1 + \frac\pi2 \frac{r^2}{\varepsilon ^2} \Big) + A\Big) + \frac1{c^2}O(R^{-2}) .
\end{align*}
Hence, combining the preceding computation with \eqref{eq:NewA} gives
\begin{align*}
\phi_\varepsilon ^2  \geqslant &-\frac1{\pi}\log \varepsilon + A_p + \frac1{2\pi}\log \frac\pi2 +\frac1{2\pi} -\frac1{\pi} \log \Big(1 + \frac\pi2 \frac{r^2}{\varepsilon ^2} \Big) + \frac1{c^2} O(R^{-2})
\end{align*}
Consequently,
\begin{align*}
2\pi \phi_\varepsilon ^2  \geqslant &- 2\log \varepsilon  + \log \frac\pi2+ 2 \pi A_p +1 -2 \log \Big(1 + \frac\pi2 \frac{r^2}{\varepsilon ^2} \Big) + \frac1{c^2} O(R^{-2}) .
\end{align*}
Using this estimate, we can integrate $\exp (2\pi\phi_\varepsilon ^2) $ over $\Om\cap B_{R\varepsilon }(p)$ to get
\begin{align*}
\int_{\Om\cap B_{R\varepsilon }(p)} e^{2\pi\phi_\varepsilon ^2} dx \geqslant &\frac{\pi}2 e^{1+2\pi A_p} \exp \big( \frac1{c^2} O(R^{-2}) \big)\varepsilon ^{-2} \int_{\Om\cap B_{R\varepsilon }(p)} \Big(1 + \frac\pi2 \frac{r^2}{\varepsilon ^2} \Big)^{-2} dx\\
= &\frac{\pi}2 e^{1+2\pi A_p} \exp \big(\frac1{c^2} O(R^{-2}) \big) \int_{B_R(p) \cap \varepsilon^{-1}(\Om-p) }\Big(1 + \frac\pi2 r^2\Big)^{-2} dx  \\
=& \frac{\pi}2 e^{1+2\pi A_p} \Big( 1 + O(R^{-2}) \Big).
\end{align*}
This combined with the estimate for $\int_{\Om\setminus B_{R\varepsilon }(p)} \exp \big( 2\pi \phi_\varepsilon ^2 \big) dx$ obtained earlier gives
\begin{align*}
\int_\Om e^{2\pi \phi_\varepsilon ^2} dx &\geqslant |\Om| + \frac{\pi}2 e^{1+2\pi A_p} +  \frac{2\pi}{c^2}\lt(\int_{\Om} \G ^2 dx  + c^2 O(R^{-2}) + O(R^{-2})\rt)\\
&=|\Om| + \frac{\pi}2 e^{1+2\pi A_p} +  \frac{2\pi}{c^2}\lt(\int_{\Om} \G ^2 dx  + O(R^{-1})\rt)
\end{align*}
Recall that $R = -\log \varepsilon$. Thus, \eqref{eq:suff} holds provided $\varepsilon >0$ is small enough. This finishes our proof.

\subsection{Proof of Corollary \ref{Main3}}

In this part, we prove \eqref{eq:Mainresult3}. Indeed, for each function $0 \not \equiv u \in H^1(\Omega)$, we set
\[
\mathop u\limits^ \circ =u - |\Omega|^{-1} \int_M u dx
\]
and let
\[
v = \mathop u\limits^ \circ (\|\nabla \mathop u\limits^ \circ \|_2^2 - \alpha \|\mathop u\limits^ \circ \|_2^2)^{-1/2} .
\]
Since $\alpha \in [0, \lambda^\NN (\Omega))$, it is not hard to see that $v$ is well-defined. Furthermore, the function $v$ satisfies $\int_\Omega v dx = 0$ and $\|v\|_{1, \alpha} =1$. Making use of \eqref{eq:Mainresult1}, there exists some uniform constant $C>0$ such that
\[
\int_\Omega e^{2\pi v^2} dx \leqslant C < +\infty.
\]
Notice that
\[
\Big( \sqrt{2\pi} v - \frac{(\|\nabla \mathop u\limits^ \circ \|_2^2 - \alpha \|\mathop u\limits^ \circ \|_2^2)^{1/2}}{2\sqrt{2\pi}}\Big )^2 \geqslant 0,
\]
which implies
\[
\mathop u\limits^ \circ - \frac 1{8\pi} (\|\nabla \mathop u\limits^ \circ \|_2^2 - \alpha \|\mathop u\limits^ \circ \|_2^2) \leqslant 2\pi v^2.
\]
Hence
\[
\int_\Omega \exp (\mathop u\limits^ \circ ) dx \leqslant C \exp \Big( \frac 1{8\pi} \int_\Omega |\nabla u|^2 dx - \frac \alpha {8\pi} \int_\Omega u^2 dx + \frac \alpha {8\pi}|\Omega|^{-1} \Big(\int_\Omega u dx \Big)^2 \Big).
\]
(Here we notice that $\|\mathop u\limits^ \circ \|_2^2 = \|u\|_2^2 - |\Omega|^{-1} \left(\int_\Omega u dx \right)^2$.) From this we obtain
\[
\int_\Om e^ u dx \leqslant C \exp \Big( \frac 1{8\pi} \int_\Omega |\nabla u|^2 dx - \frac \alpha {8\pi} \int_\Omega u^2 dx + \frac 1{|\Omega|} \int_\Omega u dx \Big)
\]
as claimed in \eqref{eq:Mainresult3}. 

\subsection{An application to a boundary value problem for mean field-type equations}
\label{subsec-Application}

In the last part of the paper, we illustrate how to use \eqref{eq:Mainresult3} by an example. Inspired by \cite[Eq. (1.2)]{CYzero}, let us consider the following Neumann boundary condition for a linear pertubation of mean field equations on domains
\begin{equation}\label{eq:Example}
\left\{
\begin{split}
-\Delta u - \alpha u &= \rho \left( \frac{ f e^u}{\int_\Omega f e^u dx} - \frac 1{ |\Omega|} \right) \qquad \text{ in } \Omega,\\
\partial_\nu u &= 0 \qquad \text{ on } \partial \Omega,
\end{split}
\right.
\end{equation}
where $\Omega \subset \R^2$ is a smooth, bounded domain with smooth boundary $\partial \Omega$. Here $\alpha$ and $\rho$ are non-negative parameters to be specified later, $f$ is a positive function, and by $\partial_\nu u$ we mean the outward normal derivative of $u$. 

We note from \eqref{eq:Example} that because the equation is no longer invariant under translation due to the linear pertubation, we can freely impose the Neumann boundary condition. To determine $\alpha$, as always, let us denote by $\lambda^\NN (\Omega)$ the first non-zero Neumann eigenvalue of $-\Delta$. %A simple variational characterization argument and standard regularity theory, it is not hard to see that there exists some smooth function $\phi$ such that
%\begin{equation}\label{eq:NeumannEigenfunction}
%\left\{\begin{split}
%-\Delta \phi &= \lambda^\NN (\Omega) \phi \qquad \text{ in } \Omega,\\
%\partial_\nu \phi &= 0 \qquad \text{ on } \partial \Omega.
%\end{split}\right.
%\end{equation}
Then we assume that $\alpha \in [0, \lambda^\NN (\Omega))$. For the parameter $\rho$, insprired by the analysis of mean field equations, we aslo focus on the interesting case $\rho >0$.

\begin{theorem}\label{thmApplication}
Suppose that $\Omega \subset \R^2$ is a smooth, bounded domain with smooth boundary $\partial \Omega$ and that $f$ is a positive function on $\Omega$. Then, for any $\alpha \in [0, \lambda^\NN (\Omega))$ and $\rho \in (0, 4\pi)$, there exists a non-trivial solution of \eqref{eq:Example}.
\end{theorem}

\begin{proof}
To look for a solution of \eqref{eq:Example}, we minimize the following energy functional
\[
F(u) = \frac 12 \int_\Omega \big( |\nabla u|^2 - \alpha u^2 \big) dx - \rho \log \Big( \int_M f e^u dx \Big) 
\]
over a close subset $\Hcirc (\Omega)$ of $H^1(\Omega)$. By the Poincar\'e inequality, it is easy to verify that $\|u\| = (\int_\Omega |\nabla u|^2 dx)^{1/2}$ is a norm on the subspace $\Hcirc (\Omega)$. A direct calculation shows that if $u$ minimizes $F$ in $\Hcirc (\Omega)$, then $u$ weakly solves
\begin{equation}\label{eq:WeakSolution}
-\Delta u - \alpha u + \mu = \rho \frac{ fe^u}{\int_\Omega fe^u dx}
\end{equation}
in $\Omega$ for some constant $\mu$ together with the boundary condition 
\[
\partial_\nu u = 0
\] 
on $\partial \Omega$. Integrating both sides of \eqref{eq:WeakSolution} over $\Omega$ gives $\mu =\rho/|\Omega|$. By standard regularity theory, we conclude that $u$ solves \eqref{eq:Example}. Thus, it suffices to show that $\inf_u F(u)$ is achieved in $\Hcirc (\Omega)$. However, thanks to \eqref{eq:Mainresult3}, there is some uniform constant $C>0$ such that
\[\begin{split}
\rho \log \Big( \int_M f e^u dx \Big) \leqslant & \frac{\rho}{4\pi} \Big( \frac 12 \int_\Omega \big( |\nabla u|^2 - \alpha u^2 \big) dx \Big) + C\rho + \rho \log f.
\end{split}\]
Thanks to $\rho < 4\pi$, we deduce that
\[\begin{split}
F(u) \geqslant & \frac 12\Big( 1 - \frac{\rho}{4\pi} \Big) \int_\Omega \big( |\nabla u|^2 - \alpha u^2 \big) dx - C\rho - \rho \log f \\
\geqslant & C^{-1} \|u\|^2 - C \rho - \rho \log f.
\end{split}\]
This implies that $F$ is bounded from below and coercive, which is enough to see that $\inf_u F(u)$ is achieved by standard arguments.
\end{proof}

\section*{Acknowledgments}

This work was done while the first author was visiting the Vietnam Institute for Advanced Study in Mathematics (VIASM). He gratefully acknowledges the institute for its hospitality and support. The second author was supported by CIMI postdoctoral research fellowship.

\section*{Appendices}

\appendix

In the following appendices, we aim to estimate $\int |\nabla w_\varepsilon |^2 dx $, $\int w_\varepsilon dx $, and $\int w_\varepsilon|^2 dx $ needed before. For convenience, let us recall that $p$ is the blow-up point and we shall use normal coordinates around $p$. Therefore, in the rest of computation, we assume that $\varepsilon$ is sufficiently small such that $\Omega \cap B_{R\varepsilon} (p)$ is the half-ball $\B_{R\varepsilon}^+(0)$ in $\Rset_+^2$ where $R= -\log \varepsilon$. We also recall the definition of $w_\varepsilon$
\[
cw_\varepsilon (r) =
\left\{
\begin{split}
\frac1{2\pi} \Big[  \log \Big(1 + \frac\pi2 R^2\Big)  - \log \Big(1 + \frac\pi2 \frac{r^2}{\varepsilon ^2}\Big)  \Big] - 2 \log (R\varepsilon ) + A_p, \; & 0 < r < R\varepsilon,\\
\G -\eta \beta,  \; & R\varepsilon \leqslant r < 2R\varepsilon,\\
\G, \; & r \geqslant 2R\varepsilon,
\end{split}
\right.
\]
and $\phi_\varepsilon (r) = w_\varepsilon (r) -|\Om|^{-1}\int_\Om w_\varepsilon dx$, where $\eta$ is cut-off function.

\section{Various estimates of $\int |\nabla w_\varepsilon |^2 dx $}
\label{apd-IntegralNablaPhi}

In this appendix, we show that 
\begin{equation}\label{eq:apdIntegralNablaPhi}
\int_\Om |\na w_\varepsilon |^2 dx = \frac1{c^2} 
\left(\alpha\int_{\Om} \G^2 dx -\frac {\log \varepsilon}{\pi} + \frac1{2\pi}\log \frac\pi2  +A_p -\frac1{2\pi} + O(R^{-2})\right).
\end{equation}
with $R= -\log \varepsilon$. To realize \eqref{eq:apdIntegralNablaPhi}, we split $\int_\Omega = \int_{\Om\cap B_{R\varepsilon }(p)} + \int_{\Om\setminus B_{R \varepsilon }(p)}$ and estimate term by term.

\subsection{Estimate of $\int_{\Om \cap B_{R \varepsilon }(p)} |\nabla w_\varepsilon |^2 dx $}
\label{apd-IntegralNablaPhi-Inner}

On the region $\Om \cap B_{R\varepsilon }(p)$, we use \eqref{eq:WInInnerBall} to get
\[
\nabla_x w_\varepsilon (r) \big|_{B_{ R\varepsilon}(p) \setminus\{p\} } =-\frac 1c \frac1{2\pi} \nabla_x  \log \lt(1 + \frac\pi2 \frac{r^2}{\varepsilon ^2}\rt)  = -\frac 1c \frac{x}{2\varepsilon^2 + \pi r^2} .
\]
From this, for small $\varepsilon$ and by Taylor's expansion we obtain
\begin{align*}
\int_{\Om \cap B_{R\varepsilon }(p)} |\na w_\varepsilon |^2 dx = &  \frac {\pi}{c^2} \int_0^{R\varepsilon} \frac{s^3}{(2\varepsilon^2 + \pi s^2 )^2} ds\\
=&  \frac1{2\pi c^2} \Big( \log (\pi R^2 +2 ) - \log 2 -  \frac{\pi R^2}{\pi R^2 +2} \Big) \\
=& \frac1{2\pi c^2} \lt(2\log R + \log \frac\pi2 -1 +O(R^{-2})\rt).
\end{align*}

%%%%%%%%%%%%%
%%%%%%%%%%%%%

\subsection{Estimate of $\int_{\Om\setminus B_{R \varepsilon }(p)} |\nabla w_\varepsilon |^2 dx $}
\label{apd-IntegralNablaPhi-Outter}

We write
\begin{align*}
c^2 \int_{\Om \setminus B_{R\varepsilon }(p)} |\na w_\varepsilon |^2 dx =&  \int_{B_{2R \varepsilon }(p) \setminus B_{R \varepsilon }(p)} |\nabla (\G -\eta \beta ) |^2 dx +  \int_{\Om\setminus B_{2R \varepsilon }(p)} |\nabla \G |^2 dx  \\
=& \int_{\Om\setminus B_{R\varepsilon }(p)} |\na \G|^2 dx- 2\int_{B_{2R\varepsilon }(p)\setminus B_{R\varepsilon }(p)}\na \G \na(\eta \beta) dx \\
& +  \int_{B_{2R\varepsilon }(p)\setminus B_{R\varepsilon }(p)}|\na (\eta \beta)|^2 dx \\
= & A_{2,1} + A_{2,2}+ A_{2,3}.
\end{align*}
Since $\beta (r) = O(r)$, there holds $\nabla (\eta \beta ) \leqslant c$ in the region $B_{2R\varepsilon }(p)\setminus B_{R\varepsilon }(p)$. Therefore, $A_{2,2}= A_{2,3}=O(R\varepsilon )$. For the term $A_{2,1}$, we now multiply both sides of the equation satisfied by $\Delta\G$ in \eqref{eq:Green} by $\G$ and integrate by parts over $\Om\setminus B_{R\varepsilon }(p)$ to get
\[\begin{split}
\int_{\Om\setminus B_{R\varepsilon }(p)} |\na \G|^2 dx =& \alpha \int_{\Om\setminus B_{R\varepsilon }(p)} \G^2 dx -\int_{\partial (\Om\setminus B_{R\varepsilon }(p))} \G \frac{\partial \G}{\partial \nu} d\sigma_x  - \frac 1{|\Omega|} \int_{\Om\setminus B_{R\varepsilon }(p)} \G dx\\
=& \alpha \int_{\Omega} \G^2 dx -\alpha \int_{B_{R\varepsilon }(p)} \G^2 dx -\int_{\partial B_{R\varepsilon }(p)\setminus \partial  \Omega } \G \frac{\partial \G}{\partial \nu} d\sigma_x\\
&\quad   +  \frac 1{|\Omega|} \int_{B_{R\varepsilon }(p)} \G dx.
\end{split}\]
Notice that in order to obtain the last step, we have used $\partial (\Om\setminus B_{R\varepsilon }(p))= [\partial \Omega \setminus \partial B_{R\varepsilon }(p)] \cup [\partial B_{R\varepsilon }(p) \setminus \partial \Omega]$ and $\partial_\nu \G \equiv 0$ on $\Omega \setminus \{ p\}$. Hence, it remains to estimate $\int_{\partial B_{R\varepsilon }(p)\setminus \partial  \Omega } \G \partial_\nu \G d\sigma_x $ and $\int_{B_{R\varepsilon }(p)} \G dx$. For the integral, $\int_{B_{R\varepsilon }(p)} \G dx$, it is not hard to see that
\[\begin{split}
\int_{B_{R\varepsilon }(p)} \G dx =& - \int_0^{R\varepsilon} s \log s ds +O((R\varepsilon)^2) + O((R\varepsilon)^3) \\
=& O((R\varepsilon )^2 \log (R\varepsilon) ) + O((R\varepsilon )^2) .
\end{split}\]
Similarly, we have
\[
\int_{B_{R\ep}(p)} \G^2 dx = O((R\varepsilon \log(R\varepsilon))^2).
\]
For the integral $\int_{\partial B_{R\varepsilon }(p)\setminus \partial  \Omega } \G \partial_\nu \G d\sigma_x $, a direct calculation leads us to
\[
\int_{\partial B_{R\varepsilon }(p)\setminus \partial  \Omega } \G \partial_\nu \G d\sigma_x =  -\frac {\log (R\varepsilon )}{\pi} + A_p + O(R\varepsilon \log (R\varepsilon )).
\]
Thus,
\[
c^2 \int_{\Om \setminus B_{R\varepsilon }(p)} |\na w_\varepsilon |^2 dx =   \alpha \int_{\Om} \G^2 dx -\frac {\log (R\varepsilon )}{\pi} + A_p + O(R\varepsilon \log (R\varepsilon )).
\]

%%%%%%%%%%%%%
%%%%%%%%%%%%%
%%%%%%%%%%%%%
%%%%%%%%%%%%%

\section{Various estimates of $\int w_\varepsilon dx $}
\label{apd-IntegralPhi}

In this appendix, we show that 
\begin{equation}\label{eq:apdIntegralPhi}
\int_\Omega w_\varepsilon dx =\frac1c O((R\varepsilon )^2 \log (R\varepsilon ))
\end{equation}
with $R= -\log \varepsilon$. As before, we also split $\int_\Omega = \int_{\Om\cap B_{R\varepsilon }(p)} + \int_{\Om\setminus B_{R \varepsilon }(p)}$ and estimate term by term.

\subsection{Estimate of $\int_{\Om \cap B_{R \varepsilon }(p)} w_\varepsilon dx $}
\label{apd-IntegralPhi-Inner}

For this integral, we estimate as follows:
\begin{align*}
c\int_{\Om\cap B_{R\varepsilon }(p)} w_\varepsilon dx =& - \frac1{2\pi} \int_{\Om\cap B_{R\varepsilon }(p)}  \log \Big(1 + \frac\pi2 \frac{r^2}{\varepsilon ^2}\Big) dx  \\
&+ \frac 12 \log \Big(1 + \frac\pi2 R^2\Big)  (R\varepsilon)^2  -  \log (R\varepsilon )(R\varepsilon)^2 + \pi A_p (R\varepsilon)^2 \\
=&- \int_0^{R\varepsilon} s  \log \Big(1 + \frac\pi2 \frac{s^2}{\varepsilon ^2}\Big) ds\\
&  + \frac 12 \log \big(\pi R^2 + 2\big)  (R\varepsilon)^2   + O((R\varepsilon )^2 \log (R\varepsilon) ) + O((R\varepsilon )^2)\\
=& \frac {\varepsilon^2}{2 \pi} \big[ -(\pi R^2 + 2)\log (\pi R^2 + 2) + (\pi R^2 + 2) \log 2 + \pi R^2 \big]\\
&  + \frac 12 \log \big(\pi R^2 + 2\big)  (R\varepsilon)^2   + O((R\varepsilon )^2 \log (R\varepsilon) ) + O((R\varepsilon )^2)\\
= & -\frac{\varepsilon^2}\pi\log (\pi R^2 + 2)    + O((R\varepsilon )^2 \log (R\varepsilon) ) + O((R\varepsilon )^2) + O(\varepsilon^2) \\
= & O((R\varepsilon )^2 \log (R\varepsilon) )  + O((R\varepsilon )^2) + O(\varepsilon^2 ).
\end{align*}

%%%%%%%%%%%%%
%%%%%%%%%%%%%

\subsection{Estimate of $\int_{\Om\setminus B_{R \varepsilon }(p)} w_\varepsilon dx $}
\label{apd-IntegralPhi-Outter}

To estimate this integral, we use the formula $\int_\Omega \G dx =0$ and the co-area formula as follows:
\begin{align*}
c\int_{\Om\setminus B_{R \varepsilon }(p)} w_\varepsilon dx= & \int_{\Om\setminus B_{2R \varepsilon }(p)} \G dx + \int_{B_{2R\varepsilon }(p)\setminus B_{R\varepsilon }(p)} (\G - \eta \beta) dx \\
=&  - \int_{\Om\cap B_{R \varepsilon }(p)} \G dx -  \int_{B_{2R\varepsilon }(p)\setminus B_{R\varepsilon }(p)} \eta \beta dx \\
=& \int_0^{R\varepsilon} s \log s ds   + (A_p + O(R\varepsilon)) \int_{\Om\cap B_{R \varepsilon }(p)} dx +   O((R\varepsilon )^2) \\
= & O((R\varepsilon )^2 \log (R\varepsilon) ) + O((R\varepsilon )^2).
\end{align*}

%%%%%%%%%%%%%
%%%%%%%%%%%%%
%%%%%%%%%%%%%
%%%%%%%%%%%%%

\section{Various estimates of $\int w_\varepsilon^2 dx $}
\label{apd-IntegralPhi^2}

In this appendix, we show that 
\begin{equation}\label{eq:apdIntegralPhi^2}
\int_\Omega w_\varepsilon^2 dx =\frac1{c^2}\left(\int_{\Om} \G ^2 dx + O((R\varepsilon )^2(\log (R\varepsilon ))^2)\right)
\end{equation}
with $R= -\log \varepsilon$. As always, we also split $\int_\Omega = \int_{\Om\cap B_{R\varepsilon }(p)} + \int_{\Om\setminus B_{R \varepsilon }(p)}$ and estimate term by term.

\subsection{Estimate of $\int_{\Om \cap B_{R \varepsilon }(p)} w_\varepsilon^2 dx $}
\label{apd-IntegralPhi^2-Inner}

We estimate this term as follows:
\begin{align*}
(2\pi)^2 c^2 \int_{\Om \cap B_{R \varepsilon }(p)} w_\varepsilon^2 dx  =& \int_{\Om \cap B_{R \varepsilon }(p)}  \Big[ \log \Big(1 + \frac\pi2 \frac{r^2}{\varepsilon ^2}\Big) \Big]^2 dx \\
& + \int_{\Om \cap B_{R \varepsilon }(p)} \Big[ \log \Big(1 + \frac\pi2 R^2\Big) \Big]^2 dx\\
& + 4 \int_{\Om \cap B_{R \varepsilon }(p)}  (\log (R\varepsilon ))^2 dx  + \int_{\Om \cap B_{R \varepsilon }(p)}  (2\pi A_p)^2 dx \\
& - 2 \int_{\Om \cap B_{R \varepsilon }(p)}   \Big[ \log \Big(1 + \frac\pi2 \frac{r^2}{\varepsilon ^2}\Big) \Big]  \Big[ \log \Big(1 + \frac\pi2 R^2\Big) \Big] dx \\
& + 4 \int_{\Om \cap B_{R \varepsilon }(p)}    \Big[ \log \Big(1 + \frac\pi2 \frac{r^2}{\varepsilon ^2}\Big) \Big] \log (R \varepsilon) dx \\
& - 4 \pi A_p \int_{\Om \cap B_{R \varepsilon }(p)} \Big[ \log \Big(1 + \frac\pi2 \frac{r^2}{\varepsilon ^2}\Big) \Big]  dx \\
& -4  \int_{\Om \cap B_{R \varepsilon }(p)}  \Big[ \log \Big(1 + \frac\pi2 R^2\Big) \Big] \log (R \varepsilon) dx \\
& + 4 \pi  \int_{\Om \cap B_{R \varepsilon }(p)}  \Big[ \log \Big(1 + \frac\pi2 R^2\Big) \Big]  dx \\
& -8 \pi A_p \int_{\Om \cap B_{R \varepsilon }(p)} \log (R \varepsilon) dx \\
=& O((R\varepsilon )^2(\log (R\varepsilon ))^2) + O((R\varepsilon )^2 \log (R\varepsilon )) + O((R\varepsilon )^2).
\end{align*}

%%%%%%%%%%%%%
%%%%%%%%%%%%%

\subsection{Estimate of $\int_{\Om\setminus B_{R \varepsilon }(p)} w_\varepsilon^2 dx $}
\label{apd-IntegralPhi^2-Outter}

We write
\begin{align*}
\int_{\Om \setminus B_{R\varepsilon }(p)} w_\varepsilon^2dx =&\frac1{c^2} \int_{B_{2R \varepsilon }(p) \setminus B_{R \varepsilon }(p)}  (\G -\eta \beta )^2 dx + \frac1{c^2}\int_{\Om\setminus B_{2R \varepsilon }(p)} \G ^2 dx  \\
=& \frac1{c^2}\int_{\Om\setminus B_{R\varepsilon }(p)} \G ^2 dx-\frac2{c^2} \int_{B_{2R\varepsilon }(p)\setminus B_{R\varepsilon }(p)} \G (\eta \beta) dx \\
& + \frac1{c^2} \int_{B_{2R\varepsilon }(p)\setminus B_{R\varepsilon }(p)} (\eta \beta)^2 dx \\
= &  \frac1{c^2}\lt(\int_{\Om\setminus B_{R\varepsilon }(p)} \G ^2 dx + C_1+ C_2\rt).
\end{align*}
Since $\eta$ is bounded, $\beta (r) = O(r)$, and $R\varepsilon \searrow 0$ as $\varepsilon \to 0$, we deduce that $C_2=O((R\varepsilon)^2)$. For the term $C_1$, we estimate as follows
\[\begin{split}
C_1 =& O(R\varepsilon)  \int_{B_{2R\varepsilon }(p)\setminus B_{R\varepsilon }(p)} \big(  \log |x| + A_p + |\beta (x)| \big) dx\\
=& O(R\varepsilon) \int_{R\varepsilon}^{2R\varepsilon} s \log s ds + O((R\varepsilon)^3) + O((R\varepsilon)^4)\\
=&O((R\varepsilon)^3 \log (R\varepsilon) ) + O((R\varepsilon)^3) + O((R\varepsilon)^4).
\end{split}\]
Hence we have just shown that
\begin{align*}
\int_{\Om \setminus B_{R\varepsilon }(p)} w_\varepsilon^2dx & =\frac1{c^2}\lt(\int_{\Om\setminus B_{R\varepsilon }(p)} \G ^2 dx + O((R\varepsilon)^2)\rt)\\
&=\frac1{c^2}\lt(\int_{\Om} \G ^2 dx + O((R\varepsilon \log(R\varepsilon))^2) + O((R\varepsilon)^2)\rt)
\end{align*}

%%%%%%%%%%%%%
%%%%%%%%%%%%%
%%%%%%%%%%%%%
%%%%%%%%%%%%%

\end{document}